\documentclass[10pt]{article}
\usepackage{amsmath}
\usepackage{amssymb}
\usepackage{stackrel,amssymb}
\usepackage{amsfonts}
\usepackage{amsthm}
\usepackage{slashed}
\usepackage{tikz-cd}
\tikzset{
    vert/.style={anchor=south, rotate=90, inner sep=.5mm}
}
\usepackage{mathrsfs}
\usepackage[all]{xy}
\usepackage{mathtools}

\usepackage{mathabx,epsfig}

\usepackage{amssymb,amsfonts}
\usepackage{enumerate}
\usepackage{enumitem}
\usepackage{mathrsfs}
\usepackage{stmaryrd}
\usepackage{hyperref}
\usepackage{cleveref}

\usepackage{relsize}
\usepackage[bbgreekl]{mathbbol}
\usepackage{amsfonts}
\DeclareSymbolFontAlphabet{\mathbb}{AMSb} 
\DeclareSymbolFontAlphabet{\mathbbl}{bbold}

\newtheorem{thm}{Theorem}[subsection]
\newtheorem{cor}[thm]{Corollary}
\newtheorem{prop}[thm]{Proposition}
\newtheorem{lem}[thm]{Lemma}

\theoremstyle{definition}
\newtheorem{defn}[thm]{Definition}

\newtheorem{exmp}[thm]{Example}

\theoremstyle{remark}
\newtheorem{rem}[thm]{Remark}

\newtheorem{warn}[thm]{Warning}

\begin{document}

\newcommand\restr[2]{{
  \left.\kern-\nulldelimiterspace 
  #1 
  \vphantom{\big|} 
  \right|_{#2} 
  }}

\makeatletter
\renewcommand{\@seccntformat}[1]{%
  \ifcsname prefix@#1\endcsname
    \csname prefix@#1\endcsname
  \else
    \csname the#1\endcsname\quad
  \fi}
\makeatother

\makeatletter
\newcommand{\colim@}[2]{%
  \vtop{\m@th\ialign{##\cr
    \hfil$#1\operator@font colim$\hfil\cr
    \noalign{\nointerlineskip\kern1.5\ex@}#2\cr
    \noalign{\nointerlineskip\kern-\ex@}\cr}}%
}
\newcommand{\colim}{%
  \mathop{\mathpalette\colim@{\rightarrowfill@\textstyle}}\nmlimits@
}
\makeatother

\newcommand\rightthreearrow{%
        \mathrel{\vcenter{\mathsurround0pt
                \ialign{##\crcr
                        \noalign{\nointerlineskip}$\rightarrow$\crcr
                        \noalign{\nointerlineskip}$\rightarrow$\crcr
                        \noalign{\nointerlineskip}$\rightarrow$\crcr
                }%
        }}%
}

\title{Completed $K$-theory and Equivariant Elliptic Cohomology}         
\author{Kiran Luecke}        
\date{\today}          
\maketitle

\begin{abstract}
Kitchloo and Morava give a strikingly simple picture of elliptic cohomology at the Tate curve by studying a completed version of $S^1$-equivariant $K$-theory for spaces. Several authors (cf [ABG],[KM],[L]) have suggested that an equivariant version ought to be related to the work of Freed-Hopkins-Teleman  ([FHT1],[FHT2],[FHT3]). However, a first attempt at this runs into apparent contradictions concerning twist, degree, and cup product. Several authors (cf. [BET],[G],[K]) have solved the problem over the complex numbers by interpreting the $S^1$-equivariant parameter as a complex variable and using holomorphicity as the technique for completion. This paper gives a solution that works integrally, by constructing a carefully completed model of $K$-theory for $S^1$-equivariant stacks which allows for certain ``convergent" infinite-dimensional cocycles. 
\end{abstract}

\tableofcontents

\section{Introduction}
In [KM], Kitchloo and Morava construct an equivariant cohomology theory for CW-spaces with an $S^1$-action by taking ordinary $S^1$-equivariant $K$-theory and completing the coefficient ring $K^{*}_{S^1}(\text{pt})\simeq\mathbb{Z}[q^{\pm}]$ in positive $q$-powers. For a finite $CW$-complex $M$ they define
$$\vec{K}^{*}_{S^1}(M):=K^{*}_{S^1}(M)\otimes_{\mathbb{Z}[q^{\pm}]}\mathbb{Z}((q))\ \ \ \ \ \ \ \ \ \ \mathbb{Z}((q)):=\mathbb{Z}[q^{-1}][[q]].$$
They extend this to a theory on infinite complexes by taking limits over finite skeleta. It satisfies a strong localization theorem for finite complexes: the inclusion of the fixed point set $j:M^{S^1}\hookrightarrow M$ induces an isomorphism
$$j^{*}:\vec{K}^{*}_{S^1}(M)\xrightarrow{\sim}\vec{K}^{*}_{S^1}(M^{S^1}).$$
Recall that $S^1$ acts on the free loop space $LM$ by loop rotation. Kitchloo and Morava show that the assignment
$$M\mapsto\vec{K}^{*}_{S^1}(LM)$$
 defines a cohomology theory. The localization theory together with the formula $LM^{S^1}=M$ shows that as a cohomology theory the above is just $K$-theory with coefficients in $\mathbb{Z}((q))$. But their insight is that this construction naturally gives more than a cohomology theory. Recall that a multiplicative cohomology theory $E$ is called an \textit{elliptic cohomology theory}\footnote{This definition is due to Mike Hopkins in his 1994 ICM address.} if it is
 
 \begin{enumerate}[itemsep=0pt]
\item weakly even: $E^{2}(\text{pt})\otimes_{E^{0}(\text{pt})} E^{n}(\text{pt})\rightarrow E^{n+2}(\text{pt})$  is a isomorphism for all $n$; 
\item comes with the data of an elliptic curve $\mathcal{E}$ over $E^{0}(\text{pt})$;
\item comes with an isomorphism of formal groups Spf$E^{0}(\mathbb{CP}^{\infty})\rightarrow\hat{\mathcal{E}}$, where the first object carries the group structure induced by the map $\mathbb{CP}^{\infty}\times\mathbb{CP}^{\infty}\rightarrow\mathbb{CP}^{\infty}$ classifying the tensor product of line bundles, and the second object is the formal completion of $\mathcal{E}$ at the identity.
 \end{enumerate}
The key to extracting an elliptic theory is the identification of a natural complex orientation coming from the Atiyah-Bott-Shapiro \textit{spin orientation} of the normal bundle of the fixed point locus $M\hookrightarrow LM$. The elliptic curve relevant to the Kitchloo-Morava construction is the \textit{Tate curve}, an elliptic curve over $\mathbb{Z}((q))$. Thus
$$Ell_{\text{Tate}}^{*}(M)\simeq\vec{K}^{*}_{S^1}(LM).$$
The simplicity of this method of producing a $K$-theoretic picture of elliptic cohomology at the Tate curve suggests that a similar approach might work in an equivariant setting. However, a few subtleties arise. 
\begin{defn}\label{loopgroupoid} For a topological groupoid $\mathbf{X}=(X_1\rightrightarrows X_0)$, define its loop groupoid $\mathcal{L}\mathbf{X}$ (sometimes called the \textit{inertia groupoid}) the free loop space object in the (2,1)-category of topological groupoids. i.e. the pullback of the diagonal $\Delta:\mathbf{X}\rightarrow\mathbf{X}\times\mathbf{X}$ along itself.. Explicitly, $\mathcal{L}\mathbf{X}$ has as its topological space of objects the set of functors $\text{pt}/\mathbb{Z}\rightarrow\mathbf{X}$ topologized as a subspace of $X_1$ (the space of morphisms of $\mathbf{X}$) and has as its space of morphisms the natural transformations of such functors, again topologized as a subspace of $X_1$.
\end{defn}
\noindent
The next step in the Kitchloo-Morava construction would be to define an equivariant theory like this
$$M/G\ \mapsto\ \vec{K}^{*}_{S^1}(\mathcal{L}(M/G)):=K^{*}_{S^1}(\mathcal{L}(M/G))\otimes_{\mathbb{Z}[q^{\pm}]}\mathbb{Z}((q)).$$
However, that does not have the desired relationship with equivariant elliptic cohomology. For example, suppose that $M$ is a point and that $G$ is simple and simply-connected. By the work of Freed, Hopkins, and Teleman ([FHT3], Theorem 5) this proposed elliptic group, suitably twisted by an element $k\in H^4(BG)\simeq\mathbb{Z}$, is concentrated in degree dim$G$, where it is isomorphic to the positive energy level $k-h$ ($h$ is the dual Coxeter number of $G$) representation ring of the semidirect product $LG\rtimes S^1$ of the loop group $LG$ with $S^1$, the latter acting on the former by loop rotation. So if dim$G$ is odd this trivially fails the first requirement of an elliptic cohomology theory, and the cup product is zero. Even if dim$G$ is even, the presence of the dual Coxeter shift violates the expectation that $k$-twisted $G$-equivariant elliptic classes of degree zero over the point are in bijective correspondence with equivalence classes of level $k$ positive energy representations of $LG\rtimes S^1$ (c.f. [Se2], [Se3]). Over the complex numbers these issues have been resolved by several authors; the common theme among them is to consider $q=e^{2\pi i\tau}$ as a complex variable on a certain moduli space and use holomorphicity as the method of completion (as opposed to the naive base-change proposed above). For example, Grojnowski's delocalized equivariant elliptic theory [G] assigns to a $G$-space $X$ a holomorphic sheaf on the moduli space of bundles over the elliptic curve, constructed by patching together local sections defined using equivariant singular cohomology with complex coefficients. The connection to positive energy representation theory is made by using the Kac character formula to identify characters of those representations with sections of the sheaf. Berwick-Evans and Tripathy [BET] have constructed a de Rham model refining Grojnowski's theory to a holomorphic sheaf of commutative differential graded algebras. For simple and simply-connected groups Kitchloo [K], in a method most similar to the one that will be presented here, uses a version of $LG$-equivariant $K$-theory built out of positive energy representations  to construct a holomorphic sheaf together with a character map from positive energy representations to sections of this sheaf. 

The purpose of the present paper is to construct a new model $E_G$ of equivariant elliptic cohomology at the Tate curve that satisfies the following conditions: 
\begin{enumerate}[itemsep=-1pt]
\item for an arbitrary compact Lie group $G$ the theory $E_G$ is defined integrally, i.e. no prime $p$ is invertible in the coefficient ring $E_G^*$
\item it admits twists $^\tau E_G$ by elements $\tau\in H^4BG$,
\item when the group $G$ is trivial $E_G$ recovers the Kitchloo-Morava theory described at the beginning of this section, 
\item the construction of $E_G$ does not reference the positive energy representation theory of loop groups in any way,
\item there is, \textit{a fortiori}, a natural map from the category $\text{Rep}_\text{pos}^\tau(LG\rtimes S^1)$ of level $\tau>0$ positive energy representations to $^\tau E_G^*(\text{pt})$ which for connected groups becomes an isomorphism of $\mathbb{Z}((q))$-modules after factoring through the Grothendieck group,
\item  and after tensoring with $\mathbb{C}$ it recovers, in a suitable sense, the previously defined theories of Grojnowski and Kitchloo. 
\end{enumerate}
As an added bonus, the construction will naturally extend to negative twists and exhibit a duality between positive and negative twist that the more imaginative reader might like to interpret as a manifestation of Serre duality on the moduli space of $G$-bundles over the Tate curve. 

\

\textit{Acknowledgements}: I would like to thank Andre Henriques, Dan Berwick-Evans, and Mike Hopkins for helpful conversations and valuable advice. I would like to thank Dan Berwick-Evans again, as well as Mikayla Kelley, for comments and corrections to previous drafts. The paper would not have been written without the guidance and support of Constantin Teleman. Many of the ideas here are due to him. Anyone who is familiar with the Freed-Hopkins-Teleman trilogy (+$\epsilon$) on twisted $K$-theory will see the shameless similarity in technique and philosophy. Lastly I would like to thank an anonymous referee for extensive and invaluable comments, suggestions, and corrections.

\

\textit{Some Definitions and Disclaimers}: Rings, abelian groups, and Hilbert spaces are assumed to be $\mathbb{Z}/2\mathbb{Z}$-graded, unless it is stated otherwise. As is standard, the \textit{loop group} $LG$ of a Lie group $G$ is the group of \textit{smooth} maps $S^1\rightarrow G$ with the topology of uniform convergence. $C_2$ is the cyclic group of order 2. Starting in Section 2, all groupoids considered in this paper are \textit{topological groupoids} (c.f. Definition \ref{topgroupoid}) unless otherwise stated. The following string of definitions sets the stage for the context in which these groupoids are considered.

\begin{defn} The category $\text{Top}$ is defined to be the category whose objects are topological spaces that are homotopy equivalent to a $CW$-complexes and whose morphisms are continuous maps.
\end{defn}

\begin{defn}\label{topgroupoid}
A \textit{topological groupoid} $\mathbf{X}=(X_{1}\rightrightarrows X_{0})$ is a groupoid object in the category $\text{Top}$. A morphism of topological groupoids is a morphism of the corresponding diagrams in $\text{Top}$.
\end{defn}

\begin{defn}
Let $\mathbf{X}=(X_{1}\rightrightarrows X_{0})$ be a topological groupoid. The \textit{coarse quotient} $[\mathbf{X}]$ is the quotient of $X_0$ under the equivalence relation defined by $x\sim y$ if there is a morphism from $x$ to $y$. In other words, it is the topological space of isomorphism classes of objects of $\mathbf{X}$.
\end{defn}

\begin{defn}\label{localequiv} A map of topological groupoids $\mathbf{X}\rightarrow\mathbf{Y}$ is a \textit{local equivalence} ([FHT1] Definition A.4) if the induced map of (discrete) groupoid-valued presheaves on $\text{Top}$ $\text{Hom}(-,\mathbf{X})\rightarrow\text{Hom}(-,\mathbf{Y})$ is an equivalence on stalks. Two topological groupoids $\mathbf{X}$ and $\mathbf{Y}$ are said to be \textit{weakly equivalent} if there is a diagram of local equivalences $\mathbf{X}\leftarrow\mathbf{Z}\rightarrow\mathbf{Y}$. Although it will not be needed, some readers may take comfort in the fact that weakly equivalent topological groupoids present the same underlying stack on the site of topological spaces (c.f. [FHT1] Remark A.5).
\end{defn}

\begin{defn}\label{localquotient}
A topological groupoid $\mathbf{X}=(X_{1}\rightrightarrows X_{0})$ is called a \textit{local quotient} groupoid if there exists a countable open cover $\{U_{\alpha}\}$ of the coarse quotient $[\mathbf{X}]$ such that the full subgroupoid associated to each $U_{\alpha}$ is weakly equivalent (c.f. Definition \ref{localequiv}) to the quotient groupoid of a Hausdorff space $M$ by a compact Lie group $K$, denoted in this paper by $M/K$.
\end{defn}

 \section{$S^1$-Completed K-theory}

In this section I construct a completed version of twisted $S^1$-equivariant $K$-theory for certain \textit{$B\mathbb{Z}$-groupoids}, which are defined as follows.
\begin{defn}\label{BZdef}
A \textit{$B\mathbb{Z}$-groupoid} is a pair\footnote{To streamline notation I reserve the right to refer to a $B\mathbb{Z}$-groupoid $(\mathbf{X},\alpha)$ by the name of the groupoid $\mathbf{X}$ and leave the $B\mathbb{Z}$-action  $\alpha$ as a mystery to be revealed as required.} $(\mathbf{X},\alpha)$ consisting of a groupoid $\mathbf{X}=(X_{1}\rightrightarrows X_{0})$ and a $B\mathbb{Z}$-action $\alpha$, i.e. an automorphism of the identity functor $\alpha: 1_{\mathbf{X}}\Rightarrow1_{\mathbf{X}}$. A morphism (``$B\mathbb{Z}$\textit{-equivariant map}") of $B\mathbb{Z}$-groupoids $(\mathbf{X},\alpha)\rightarrow (\mathbf{X}',\alpha')$ is a morphism of topological groupoids $F:\mathbf{X}\rightarrow \mathbf{X}'$ such that for every $x\in X_0$, $F(\alpha(x))=\alpha'(F(x))$. A \textit{$B\mathbb{Z}$-subgroupoid} of a $B\mathbb{Z}$-groupoid $(\mathbf{X},\alpha)$ is a $B\mathbb{Z}$-groupoid $(\mathbf{X}',\alpha')$ such that $\mathbf{X}'$ is a subgroupoid of $\mathbf{X}$ containing all components $\alpha(x')$, $x'\in X_0'\subset X_0$ and $\alpha'$ agrees with the restriction of $\alpha$ to $\mathbf{X}'$. A $B\mathbb{Z}$-groupoid $(\mathbf{X},\alpha)$ is called \textit{trivial} if $\alpha$ is the trivial automorphism of the identity.
\end{defn}

To model the `quotient' of a $B\mathbb{Z}$-groupoid by its $B\mathbb{Z}$-action I make the following definition.

\begin{defn}\label{BZquotient}
For a $B\mathbb{Z}$-groupoid $(\mathbf{X},\alpha)$ define the \textit{$B\mathbb{Z}$-quotient} $\mathbf{X}/B\mathbb{Z}$ to be groupoid whose space of objects is $X_{0}$ and whose space of morphisms is $(X_{1}\times\mathbb{R})/\mathbb{Z}$, where $\mathbb{Z}$ acts as follows: if $s$ denotes the source morphism of $\mathbf{X}$, then for $n\in\mathbb{Z}$, $n\cdot(p,r):=(p\alpha(s(p))^{-n},r+n).$ This is functorial in the $B\mathbb{Z}$-groupoid.
\end{defn}

\begin{exmp}\label{loopgroupoid}
The motivating example of a $B\mathbb{Z}$-groupoid is the \textit{loop groupoid} $\mathcal{L}(M/G)$ of a global quotient $M/G$, which is again a global quotient: its space of objects is the subspace of $M\times G$ of pairs $\{(m,g)\in M\times G|\ gm=m\}$ and $G$ acts by translation on the first factor and conjugation on the second. The $B\mathbb{Z}$-action is the automorphism of the identity functor whose component at the object $(m,g)$ is $g$. 
\end{exmp}

When $M$ is a point this is the quotient groupoid associated to $G$ acting on itself by conjugation, and when $G$ is connected the $B\mathbb{Z}$-quotient admits a (possibly enlightening) second description, up to equivalence: let $\mathcal{A}(G)$ be the space of connections on the trivial principal $G$-bundle over $S^1$ and $\mathbb{L}G$ the group of smooth bundle isomorphisms covering rigid rotations of the base $S^1$. Note that $\mathbb{L}G\simeq LG\rtimes S^1$. Then
$$\mathcal{L}(\text{pt}/G)/B\mathbb{Z}\simeq\mathcal{A}(G)/\mathbb{L}G.$$
The map inducing the equivalence goes from right to left and is given by taking the holonomy of a connection (c.f. [FHT2] Section 2.1).

\subsection{Central extensions and twists}

This subsection describes the model for \textit{twists} used in this paper.
\begin{defn} (c.f. [FHT1] Section 2.2) A groupoid is said to be \textit{graded} if it is equipped with a functor  $\epsilon$ to $\text{pt}/C_2$. A \textit{graded central extension} $\text{pt}/U(1)\rightarrow\mathbf{L}\rightarrow\mathbf{X}$ of a groupoid $\mathbf{X}=(X_{1}\rightrightarrows X_{0})$ is a graded groupoid $\mathbf{L}=(L\rightrightarrows X_{0},\epsilon)$ and a functor $P:\mathbf{L}\rightarrow\mathbf{X}$ which is the identity map on objects and is such that the induced map $P_{1}:L\rightarrow X_{1}$ is a principal $U(1)$-bundle.
The category $\mathfrak{Ext}_{\mathbf{X}}$ is defined to have objects the graded central extensions of $\mathbf{X}$, and a morphism $(\mathbf{L}_1,\epsilon_1)\rightarrow(\mathbf{L}_2,\epsilon_2)$ is the following data: a pair $(M,\eta)$ of an isomorphism class of principal $U(1)$-bundle $M\rightarrow X_{0}$, an isomorphism of $U(1)$-bundles $t^{*}M\otimes L_1\otimes s^{*}M^{-1}\rightarrow L_2$, and a continuous function $\eta:X_{0}\rightarrow C_2$ such that $\epsilon_2=t^{*}\eta\epsilon_2 s^{*}\eta^{-1}$.
\end{defn}

\begin{exmp}\label{torustwist} Let $T=U(1)^{\times r}$ be a torus. Consider the quotient groupoid $T/T=(T\times T\rightrightarrows T$ associated to the (trivial) conjugation action of $T$ on itself. Let $\tau$ be a homomorphism $\tau:\pi_1T\rightarrow \Lambda:=\text{Hom}_\text{Grp}(T,U(1))$. Define a $\pi_1T$-action on $\mathfrak{t}\times T\times U(1)$ by the formula $p\cdot(X,t,e^{i\theta})=(X+p,t,e^{i\theta}\tau(p)(t))$ and let $L^\tau$ be the quotient $(\mathfrak{t}\times T\times U(1))/\pi_1T$. Then $L^\tau\rightrightarrows T$ is a groupoid with source and target map both given by $([X,t,e^\theta)]\mapsto e^X$, and the evident morphism to $T\times T\rightrightarrows T$ is a graded central extension once we equip both groupoids with the trivial grading.
\end{exmp}

\begin{rem} The category of graded central extensions is extremely sensitive to the groupoid presentation $\mathbf{X}$ of the underlying stack. For example, suppose that the underlying stack is equivalent to a finite $CW$-complex $X$. In the presentation of $X$ as $X\rightrightarrows X$, all objects of $\mathfrak{Ext}_{\mathbf{X}}$ are trivial as principal $U(1)$-bundles (the identity morphisms provide a section) and are therefore determined by their grading, so that $\pi_{0}\mathfrak{Ext}_{\mathbf{X}}=H^{0}(X;C_2)$. On the other hand, if $X$ is presented as the groupoid $\coprod U_{ij}\rightrightarrows \coprod U_{i}$ associated to a good Cech cover of $X$, then any Cech $1$-cocycle with values in line bundles defines a central extension. It follows that $\pi_{0}\mathfrak{Ext}_{\mathbf{X}}=H^{1}(X;C_2)\times H^{1}(X;BU(1))=H^{1}(X;C_2)\times H^{3}(X)$. This is in accordance with the fact that there is a general homotopy theoretic framework for twisted cohomology theories which foresees the possibility of twisting $K^{*}(X)$ by classes in $H^3(X)$.
\end{rem}

\begin{rem}\label{deg3twist}
 The appearance of $H^3$ in the above is not a coincidence - for all groupoids $\mathbf{X}$ considered in this paper, the set of isomorphism classes of graded central extension whose grading is trivial is isomorphic to the kernel of $H^3\mathbf{X}\rightarrow H^3X_0$ (c.f. [FHT1] Proposition 2.13). 
\end{rem}

\begin{defn}
In light of the above, define the category $\mathfrak{Twist}_{\mathbf{X}}$ as follows. Its objects are \textit{twists}: a pair $(\mathbf{V},\mathbf{L}^{\tau})$ consisting of a local equivalence (c.f. Definition \ref{localequiv}) $\mathbf{V}\rightarrow\mathbf{X}$ and a graded central extension $\mathbf{L}^{\tau}$ of $\mathbf{V}$. A morphism of twists from $(\mathbf{V},\mathbf{L}^{\tau})$ to $(\mathbf{U},\mathbf{L}^{\sigma})$ is a local equivalence $\phi:\mathbf{V}\rightarrow\mathbf{U}$ over $\mathbf{X}$ and a morphism in $\mathfrak{Ext}_{\mathbf{V}}$ from $\mathbf{L}^{\tau}$ to $\phi^{*}\mathbf{L}^{\sigma}$. A map of groupoids $f:\mathbf{X}\rightarrow\mathbf{Y}$ induces a pullback functor $f^*:\mathfrak{Twist}_{\mathbf{Y}}\rightarrow \mathfrak{Twist}_{\mathbf{X}}$.
\end{defn}
\begin{defn}

If $\mathbf{X}$ is a $B\mathbb{Z}$-groupoid, then the category $B\mathbb{Z}\text{-}\mathfrak{Twist}_{\mathbf{X}}$ of \textit{equivariant} twists is defined to be the category $\mathfrak{Twist}_{\mathbf{X}/B\mathbb{Z}}$. Define the category $B\mathbb{Z}\text{-}\mathfrak{Twist}$ whose objects are pairs consisting of a $B\mathbb{Z}$-groupoid $\mathbf{X}$ and an equivariant twist $(\mathbf{V},\mathbf{L}^{\tau})\in B\mathbb{Z}\text{-}\mathfrak{Twist}_{\mathbf{X}}$. A morphism from 
$(\mathbf{X},(\mathbf{V},\mathbf{L}^{\tau}))$ to $(\mathbf{Y},(\mathbf{U},\mathbf{L}^{\sigma}))$ is a $B\mathbb{Z}$-equivariant map $\varphi:\mathbf{X}\rightarrow\mathbf{Y}$ and a morphism from $(\mathbf{V},\mathbf{L}^{\tau})$ to $(\varphi^{*}\mathbf{U},\varphi^{*}\mathbf{L}^{\sigma})$ in $B\mathbb{Z}\text{-}\mathfrak{Twist}_{\mathbf{X}}$. A \textit{homotopy} between two such morphisms is defined to be a $B\mathbb{Z}$-equivariant map $\mathbf{X}\times [0,1]\rightarrow\mathbf{Y}$ restricting to each of the $B\mathbb{Z}$-equivariant maps $\mathbf{X}\rightarrow\mathbf{Y}$ at the endpoints.
\end{defn}

\begin{defn}
Let $B\mathbb{Z}\text{-}\mathfrak{Twist}_\text{rel}$ be the category whose objects are triples $(\mathbf{X},\mathbf{A},(\mathbf{V},\mathbf{L}^{\tau}))$ consisting of a $B\mathbb{Z}$-groupoid $\mathbf{X}$, a full $B\mathbb{Z}$-subgroupoid $\mathbf{A}$, and an object $(\mathbf{V},\mathbf{L}^{\tau})\in B\mathbb{Z}\text{-}\mathfrak{Twist}_{\mathbf{X}}$. Morphisms are the relative versions of those in $B\mathbb{Z}\text{-}\mathfrak{Twist}$. The relative twisted $\vec{K}_{S^1}$-theory spectrum $^{\tau}\vec{\mathcal{K}}_{S^1}(\textbf{X},\mathbf{A})$ is defined to be the homotopy fiber of the map induced by the inclusion $\mathbf{A}\rightarrow \mathbf{X}$, and its negatively indexed homotopy groups define the relative twisted $\vec{K}_{S^1}$-theory groups $^{\tau}\vec{K}^{*}_{S^1}(\textbf{X},\mathbf{A})$ associated to the object $(\mathbf{X},\mathbf{A},(\mathbf{V},\mathbf{L}^{\tau}))$.

\end{defn}

\subsection{The cohomology theory}

In this subsection a certain $\Omega$-spectrum is constructed for each object $(\mathbf{X},(\mathbf{V},\mathbf{L}^{\tau}))$ of $B\mathbb{Z}\text{-}\mathfrak{Twist}$. Let $G_{x}$ and $\tilde{G}_{x}$ denote the automorphism groups of an object $x$ in $\textbf{X}$ and $\mathbf{X}/B\mathbb{Z}$. By Definition \ref{BZquotient} there is an exact sequence 
$$1\rightarrow G_{x}\hookrightarrow\tilde{G}_{x}=(G_{x}\times\mathbb{R})/\mathbb{Z}\rightarrow S^1\rightarrow1.$$
where $\mathbb{Z}$ sits inside $G_x\times\mathbb{R}$ as the subgroup $\{(\alpha(x)^{-n},n)\}$. Because $G_{x}$ is compact the sequence admits a fractional right-splitting, i.e. a splitting after replacing $S^1$ by a finite cover $S^1_d\rightarrow S^1$. 
The construction of the $\vec{K}_{S^1}$-theory spectrum will require a choice of fractional splitting $\psi_x:  \tilde{G}_{x}\leftarrow S^1_d$ but the final product will not depend on it. Note that the space of choices of $\psi_x$ is a torsor for the group $\text{Hom}(S^{1},G_{x})^\text{conj}$ of conjugacy classes of homomorphisms $S^1\rightarrow G_x$, and that a global choice may not exist over all of $\mathbf{X}/B\mathbb{Z}$.
An example of this situation is the $B\mathbb{Z}$-action on $\mathbf{X}=\mathcal{L}(\text{pt}/U(1))\simeq U(1)/U(1)$ (c.f. Example \ref{loopgroupoid}). Then $\mathbf{X}/B\mathbb{Z}$ is equivalent to a $U(1)$-gerbe over $U(1)\times (\text{pt}/B\mathbb{Z})$ whose Dixmier-Douady invariant is a generator of $H^{3}(U(1)\times BS^1)$. Therefore a global fractional splitting---which is equivalent to a trivialization of the gerbe classified by some (nonzero) multiple of the generator---does not exist.\footnote{The appendix contains a treatment of this example in greater detail.}

Ordinary $K$-theory is concerned with equivalence classes of finite-dimensional vector bundles. Morally, local $\mathbb{Z}((q))$-completion corresponds to relaxing this condition to allow for infinite-dimensional bundles, provided that the image of $S^1$ under $\psi_x$ acts on the fiber with finite-dimensional isotypic subspaces and that the set of irreducible characters of $S^{1}$ that appear is bounded below as a subset of $\mathbb{Z}$. The failure of $\psi_x$ to extend over $X_{0}$ makes this ill-defined: the same bundle can meet these requirements for one choice of $\psi_x$ and fail them for another. Luckily the naive fix---restricting to those bundles whose fibers satisfy the above property for all choices of $\psi_x$---turns out to work nicely.

I work in the Freed-Hopkins-Teleman model of twisted $K$-theory for local quotient groupoids. The representing spectrum will be constructed from spaces of sections of bundles of Fredholm operators. Let $Cl_{n}$ be the Clifford algebra on $\mathbb{C}^{n}$, generated by the standard basis elements $e_{1},...,e_{n}$. For a Hilbert space $H$ let $\mathcal{B}(H)$ denote the spaces of bounded operators with the compact-open topology and let $\mathcal{K}(H)$ denote the spaces of compact operators with the norm topology. For odd $n$ and an operator $A\in\mathcal{B}(Cl_{n}\otimes H)$ define
$$w(A):=
\begin{cases}
e_{1}...e_{n}A\ \ \ \ n\equiv-1\ \text{mod}\ 4\\
-ie_{1}...e_{n}A\ \ \ \ n\equiv 1\ \text{mod}\ 4. 
\end{cases}$$
 Let $\text{Fred}^{n}(H)$ be the set of odd skew-adjoint Fredholm operators $A\in\mathcal{B}(Cl_{n}\otimes H)$ such that $A$ commutes with $Cl_{n}$, $A^2+1$ is a compact operator, and, when $n$ is odd, the operator $w(A)$ has both positive and negative eigenvalues. This set is topologized via the inclusion
 $$\text{Fred}^{n}(H)\hookrightarrow\mathcal{B}(H)\times\mathcal{K}(H)$$
 $$A\mapsto (A, A^2+1).$$
 
\begin{defn}
Let $(\mathbf{X},(\mathbf{V},\mathbf{L}^{\tau}))$ be an object of $B\mathbb{Z}\text{-}\mathfrak{Twist}$. Recall that there is an implicit grading $\epsilon: \mathbf{L}^{\tau}\rightarrow\text{pt}/C_2$. A \textit{$\tau$-twisted Hilbert bundle} over $\mathbf{X}$ is a Hilbert bundle $\mathcal{H}\rightarrow\mathbf{L}^{\tau}$ such that for any object $x$ in $\mathbf{L}^{\tau}$ the central $U(1)\in\text{Aut}(x)$ acts on the fiber $\mathcal{H}_x$ by scalar multiplication, and for any morphism $f:x\to y$  the map of fibers $\mathcal{H}_x\rightarrow\mathcal{H}_y$ has degree $\epsilon(f)$.
\end{defn}
 
\begin{defn}
Let $(\mathbf{X},(\mathbf{V},\mathbf{L}^{\tau}))$ be an object of $B\mathbb{Z}\text{-}\mathfrak{Twist}$. A $\tau$-twisted Hilbert bundle $\mathcal{H}\rightarrow \textbf{X}$ over a groupoid over $\textbf{X}$ is said to be \textit{locally universal} if for every open subgroupoid $U\hookrightarrow \textbf{X}$ and every $\tau$-twisted Hilbert bundle $\mathcal{V}$ over $U$, there exists a unitary embedding of  $\mathcal{V}$ into the restriction $\mathcal{H}_{U}$ of $\mathcal{H}$ to $U$. If $\mathbf{X}$ is a local quotient groupoid (c.f. Definition \ref{localquotient}) then a locally universal Hilbert bundle always exists and is unique up to unitary equivalence (c.f. [FHT1] Lemma 3.12).
\end{defn}

\begin{defn} Let $(\mathbf{X},(\mathbf{V},\mathbf{L}^{\tau}))$ be an object of $B\mathbb{Z}\text{-}\mathfrak{Twist}$ and let $\mathcal{H}\rightarrow\mathbf{L}^{\tau}$ be a locally universal $\tau$-twisted Hilbert bundle. Then the (uncompleted) $S^{1}$-equivariant, $\tau$-twisted $K$-theory of $\textbf{X}$ is defined in [FHT1] (Section 3.4, A.5) to be
$$^\tau K_{S^{1}}^{n}(\textbf{X}):=\ ^\tau K^{n}(\mathbf{X}/B\mathbb{Z}):=\pi_{k}\Gamma(\mathbf{L}^{\tau}; \text{Fred}^{n+k}(\mathcal{H})).$$
 It is independent of the choice of $\mathcal{H}$ (c.f. [FHT1] Remark 3.17).
 \end{defn}
 
Over an object $x$ the fiber $\mathcal{H}_{x}$ is isomorphic to $L^2(\tilde{G}_{x})\otimes\ell^2\otimes Cl_{1}$ (c.f. [FHT1] Lemma A.32). A choice of fractional splitting $\psi_x:\tilde{G}_{x}\leftarrow S^1_d$ gives an isomorphism $\mathcal{H}_{x}\simeq L^2(G_{x})\otimes L^2(S^1)\otimes\ell^2\otimes Cl_{1}$. An irreducible representation of $S^1$ is labelled by an integer $k\in\mathbb{Z}$ and has character $q\mapsto q^k$. For $F\in\mathcal{B}(\mathcal{H}_{x})^{S^1}$, let $F_{\psi_x}(k)$ denote the restriction of $F$ to the $S^1$-isotypic component of weight $k$, after splitting $\mathcal{H}_{x}\simeq V_{\psi_x}\otimes L^2(S^1)$ using $\psi_x$ as above. 
 
\begin{defn}\label{qfred}
Define the set of $\psi_x$-\textit{relative $q$-Fredholm operators} as
$$q_{\psi_x}\text{Fred}^{n}(\mathcal{H}_{x}):=\{F\in\mathcal{B}(\mathcal{H}_{x})^{S^1}| F_{\psi_x}(k)\in\text{Fred}^{n}(V_{\psi_x})^{G_{x}}\ \text{invertible for}\ k<<0\}.$$
Note that in particular each $F_{\psi_x}$ commutes with the $G_x$-action. As remarked earlier, this set depends on $\psi_x$. Define the set of \textit{$q$-Fredholm operators} $q\text{Fred}^{n}(\mathcal{H}_{x})$ to be
the intersection of  $q_{\psi_x}\text{Fred}^{n}(\mathcal{H}_{x})$
for all choices of $\psi_x$ (recall that the set of such choices is a torsor for $\text{Hom}(S^1,G_x)^\text{conj}$). Finally, topologize this set via the inclusion
\begin{align*}
   q\text{Fred}^{n}(\mathcal{H}_{x})& \hookrightarrow\prod_{\gamma\in\text{Hom}(S^1,G_{x})}\prod_{k\in\mathbb{Z}}\text{Fred}^{n}(V_{\gamma\cdot\psi_x})\\
F&\mapsto \prod\prod F_{\gamma\cdot\psi_x}(k).
\end{align*}
Although the inclusion depends on $\psi_x$, the induced topology does not, because for any other splitting $\psi'_x$, $V_{\psi'_x}$ is unitarily equivalent to $V_{\psi_x}$, so changing the splitting amounts to shifting the $\gamma$ index in the product.
\end{defn}

\begin{rem}
This definition is capturing the more humanly comprehensible idea that a $q$Fredholm operator on $\mathcal{H}_x$ is an operator which, under \textit{every} decomposition $\mathcal{H}_x\simeq V_{\psi_x}\otimes L^2S^1$, `looks like' a Laurent series (in the monomials $q^k$ that label $S^1$-representations) of Fredholm operators on $V_ {\psi_x}$.
\end{rem}

\begin{exmp}\label{rank1example} (c.f. Examples \ref{loopgroupoid} and \ref{torustwist})
Let $\mathbf{X}$ be the action groupoid of the trivial action of $U(1)$ on itself $U(1)/U(1)\simeq \mathcal{L}(\text{pt}/U(1))$ which is a $B\mathbb{Z}$-groupoid with $B\mathbb{Z}$-action given by the automorphism whose component at an object $t$ is the morphism determined by $t$. It admits a $B\mathbb{Z}$-equivariant twist for every $\tau\in\mathbb{Z}\simeq H^3_{S^1}(U(1)/U(1))$. For each such $\tau$ the twist can be presented explicitly as the pair $(\mathbf{V},\mathbf{L}^\tau)$ where $\mathbf{V}\simeq\mathbf{X}/B\mathbb{Z}$ and $\mathbf{L}^\tau$ is described as follows. 
Define an action of $\pi_1U(1)=\mathbb{Z}$ on $i\mathbb{R}\times U(1)\times U(1)\times\mathbb{R}$ by the formula (c.f. Definition \ref{explicitmodelTtau})
$$p\cdot (iX,e^{2\pi i\theta},e^{2\pi i\phi},r)=(i(X+p),e^{2\pi i(\theta-p\tau X+\tau p^2r)},e^{2\pi i(\phi+pr)},r).$$
There is a commuting $\mathbb{Z}$-action (c.f. Definition \ref{BZquotient}) defined by the formula 
$n\cdot(iX,e^{2\pi i\theta},e^{2\pi i\phi},r)=(iX,e^{2\pi i\theta},e^{2\pi i(\phi-nX)},r+n)$
and write $L^\tau$ for the quotient by $\pi_1T\times\mathbb{Z}$. Define the groupoid $\mathbf{L}^\tau=(L^\tau\rightrightarrows U(1))$ by declaring the source and target map to both be $[(X,\theta,t,r)]\mapsto e^{2\pi iX}$. The set of fractional splittings at an object $t$ is isomorphic to the set of $iX$ such that $e^{2\pi i X}=t$ (which is a torsor for $\text{Hom}(S^1,U(1))\simeq\mathbb{Z}$) since each such lift induces an identification of the automorphism group of $t$ with $U(1)\times U(1)\times S^1$ by translating with the $\pi_1U(1)$-action until $iX$ is in the fundamental domain containing 0.  Fixing such a fractional splitting $\psi_t$, a $q_{\psi_t}$Fredholm operator $F$ is represented by a Laurent series $\sum_kF(k)q^k$ where $F(k)$ is a $U(1)$-invariant Fredholm operator on $L^2U(1)$. In other words, $F$ can be represented by a Laurent series $\sum\chi_kq^k\in R(U(1))((q))\simeq\mathbb{Z}[t^\pm]((q)))$. Write $R(U(1))\simeq\mathbb{Z}[t\pm]$. From the explicit presentation of $\mathbf{L}^\tau$ in Example \ref{torustwist}, changing the splitting $\psi_t$ by the generator $1\in\text{Hom}(S^1,U(1))=\pi_1U(1)$ sends $\sum\chi_k(t)q^k$ to $(qt)^\tau\sum\chi_k(qt)q^k$. In order for the latter to be a Laurent series, it must be that $\chi_k(t)\sim O(t^k)$ as $k\to\infty$.
\end{exmp}

Letting $x$ vary over the space of objects of $\mathbf{L}^{\tau}$ in the above construction defines, for each $n$, a subspace $q\text{Fred}^{n}(\mathcal{H})$ of the bundle $\mathcal{B}(\mathcal{H})\times\mathcal{K}(\mathcal{H})\rightarrow\mathbf{L}^{\tau}$. In general $ q\text{Fred}^{n}(\mathcal{H})\rightarrow \mathbf{L}^{\tau}$ may not be a bundle or even a fibration. Nevertheless, it is a continuous surjection and therefore has an associated sheaf of continuous sections, denoted by $q\mathcal{F}^{n}(\mathcal{H})$. The following is an easy consequence of the methods developed in the Appendix of [FHT1].

\begin{lem} There is a weak homotopy equivalence $\Gamma(\mathbf{L}^{\tau}; q\mathcal{F}^{n+1}(\mathcal{H}))\xrightarrow{\sim} \Omega\Gamma(\mathbf{L}^{\tau}; q\mathcal{F}^{n}(\mathcal{H}))$.
\end{lem}

Because $Cl_{2}\simeq\text{M}_{2}(\mathbb{C})$ is a matrix algebra there is a homeomorphism 
$$\Gamma(\mathbf{L}^{\tau}; q\mathcal{F}^{n+2}(\mathcal{H}))\rightarrow\Gamma(\mathbf{L}^{\tau}; q\mathcal{F}^{n}(\mathcal{H})).$$

\begin{defn}\label{Khatdef}
 Let $(\mathbf{X},(\mathbf{V},\mathbf{L}^{\tau}))$ be an object of $B\mathbb{Z}\text{-}\mathfrak{Twist}$. Define the $\tau$-twisted $\vec{K}_{S^1}$-theory $\Omega$-spectrum of $\textbf{X}$ as:

\[ \vec{\mathcal{K}}_{S^1}(\textbf{X})_{n}:=\begin{cases} 
      \Gamma(\mathbf{L}^{\tau}; q\mathcal{F}^{0}(\mathcal{H})) & n\ \text{even}\\
      \Gamma(\mathbf{L}^{\tau}; q\mathcal{F}^{1}(\mathcal{H})) & n\ \text{odd}
      \end{cases}
\]
with the structure maps $\vec{\mathcal{K}}_{S^1}(\textbf{X})_{n}\rightarrow\Omega\vec{\mathcal{K}}_{S^1}(\textbf{X})_{n+1}$ given by a combination of the previous two maps, as needed for the parity of $n$. Suppose $\mathbf{A}\subset\mathbf{X}$ is a full $B\mathbb{Z}$-subgroupoid (c.f. Definition \ref{BZdef}). The pullback of $(\mathbf{V},\mathbf{L}^{\tau})$ to $\mathbf{A}$ defines a ($B\mathbb{Z}$-equivariant) twist $(\mathbf{V}_\mathbf{A},\mathbf{L}^{\tau}_\mathbf{A})$ of $\mathbf{A}$, and the pullback of $\mathcal{H}$ to $\mathbf{L}^{\tau}_\mathbf{A}$ is $\tau$-twisted and locally universal(c.f. [FHT1] Corollary A.34). Therefore the restriction map $\iota_{\mathbf{A}}^{n}:\Gamma(\mathbf{L}^{\tau}; q\mathcal{F}^{n}(\mathcal{H}))\rightarrow \Gamma(\mathbf{L}^{\tau}_\mathbf{A}/B\mathbb{Z}; q\mathcal{F}^{n}(\mathcal{H}))$ is well-defined. The relative $\vec{K}_{S^1}$-theory $\Omega$-spectrum $\vec{\mathcal{K}}_{S^1}(\textbf{X},\mathbf{A})$ is defined to be its homotopy fiber. Finally, define

$$\vec{K}_{S^1}^{*}(\textbf{X},\mathbf{A}):=\pi_{-*}\vec{\mathcal{K}}_{S^1}(\textbf{X},\mathbf{A}).$$
\end{defn}

Of course, at this point it is not clear how much the definition depends on the choice of twist $(\mathbf{V},\mathbf{L}^\tau)$, but the following lemma settles the question: the homotopy type of the spectrum depends only on the isomorphism class of the twist, which is an element in $H^1(\mathbf{X};C_2)\oplus H^3\mathbf{X}$.

\begin{prop}For $n\in\mathbb{Z}$, the collection of assignments 
$$(\mathbf{X},\mathbf{A},(\mathbf{V},\mathbf{L}^{\tau}))\mapsto\ ^{\tau}\vec{K}_{S^1}^{n}(\mathbf{X},\mathbf{A})$$ forms a twisted cohomology theory. More precisely,
\begin{enumerate}[itemsep=0pt]
\item[i)] this defines a contravariant functor from $B\mathbb{Z}\text{-}\mathfrak{Twist}_\text{rel}$ to $\mathbb{Z}((q))\text{-mod}$ taking local equivalences to isomorphisms and taking homotopic\footnote{Homotopy is defined by the standard interval object $[0,1]$.} morphisms to equal ones;

\item[ii)] there is a natural long exact sequence
$$...\rightarrow\ ^{\tau}\vec{K}_{S^1}^{n}(\mathbf{X},\mathbf{A})\rightarrow\ ^{\tau}\vec{K}_{S^1}^{n}(\mathbf{X})\rightarrow\ ^{\tau}\vec{K}_{S^1}^{n}(\mathbf{A})\rightarrow\ ^{\tau}\vec{K}_{S^1}^{n+1}(\mathbf{X},\mathbf{A})\rightarrow...$$

\item[iii)] (excision) if $\mathbf{Z}\subset\mathbf{A}$ is a full $B\mathbb{Z}$-subgroupoid whose closure is contained in the interior of $\mathbf{A}$, then the restriction map 
$$^{\tau}\vec{K}_{S^1}^{n}(\mathbf{X},\mathbf{A})\longrightarrow\ ^{\tau}\vec{K}_{S^1}^{n}(\mathbf{X}\setminus\mathbf{Z},\mathbf{A}\setminus\mathbf{Z})$$
is an isomorphism;

\item[iv)] if $J$ is an index set and $(\mathbf{X},\mathbf{A},(\mathbf{V}, \pmb{\tau}))=\coprod_{J}(\mathbf{X}_{j},\mathbf{A}_{j},(\mathbf{V}_{j}, \mathbf{L}^{\tau_{j}}))$ is a disjoint union, then 
$$^{\tau}\vec{K}_{S^1}^{n}(\mathbf{X},\mathbf{A})\longrightarrow\ \prod_{J}\ ^{\tau_{j}}\vec{K}_{S^1}^{n}(\mathbf{X}_{j},\mathbf{A}_{j})$$
is an isomorphism.
\end{enumerate}
\end{prop}
\begin{proof} This is essentially the proof given in Section 3.5 of [FHT1] with minor changes. Functoriality is immediate from the construction of the spectrum. Write $I=[0,1]$ for the unit interval. Homotopy invariance follows from the fact that if $\mathcal{H}\rightarrow\mathbf{L}^{\tau}$ is a locally universal $\tau$-twisted Hilbert bundle and $p:\mathbf{L}^{\tau}\times I \rightarrow\mathbf{L}^{\tau}$ is the projection then $p^{*}\mathcal{H}\rightarrow\mathbf{L}^{\tau}\times I$ is $\tau$-twisted\footnote{To be pedantic, $p^*\tau$-twisted.} locally universal ([FHT1] Lemma A.32), so
$$^{\tau}\vec{\mathcal{K}}_{S^1}\big((\textbf{X},\mathbf{A})\times I\big)_{n}\simeq\vec{\mathcal{K}}_{S^1}(\textbf{X},\mathbf{A})_{n}^{I},$$
making the two restriction maps homotopic.

The fact that local equivalences are taken to isomorphisms is a consequence of \textit{descent} (c.f. [FHT1] Lemma A.18), which states that the pullback $f^*$ along a local equivalence $f:\mathbf{X}\rightarrow\mathbf{Y}$ induces an equivalence of categories, from groupoids over $\mathbf{Y}$ to groupoids over $\mathbf{X}$, with a natural adjoint inverse denoted $f_*$. Hence for any  $\mathbf{P}\rightarrow\mathbf{Y}$ the natural map from $\Gamma(\mathbf{Y},{P})\rightarrow\Gamma(\mathbf{X},f^*\mathbf{P})$ is a homeomorphism whose inverse is the composition of the natural map $\Gamma(\mathbf{X},f^*\mathbf{P})\rightarrow\Gamma(\mathbf{Y},f_*f^*\mathbf{P})$ with the map on sections induced by counit $f_*f^*\mathbf{P}\rightarrow \mathbf{P}$.

The long exact sequence in $\textit{i)}$ is obtained from the fiber sequence
$$...\rightarrow\Omega^{\tau}\vec{\mathcal{K}}_{S^1}(\mathbf{A})_{n}\rightarrow\ ^{\tau}\vec{\mathcal{K}}_{S^1}(\textbf{X},\mathbf{A})_{n}\rightarrow \ ^{\tau}\vec{\mathcal{K}}_{S^1}(\textbf{X})_{n}\rightarrow\ ^{\tau}\vec{\mathcal{K}}_{S^1}(\mathbf{A})_{n}\rightarrow...$$
The claim $\textit{iv)}$ about disjoint unions is immediate from the definition. It remains to prove excision. Despite some cumbersome notation, the proof use a few standard homotopy-theoretic constructions to boil things down to the following fact ([FHT1] Lemma A.32): the pullback of a $\tau$-twisted locally universal Hilbert bundle over a local quotient groupoid to a full subgroupoid is again locally universal. Amusingly, the use of this fact will make its appearance in a footnote.
Let 
$\mathbf{M}=\mathbf{X}\setminus\mathbf{Z}\cup I\times \big(\mathbf{A}\setminus\mathbf{Z}\big)\cup(\mathbf{A}\setminus\mathbf{Z}\big)\times I\cup\mathbf{A}$ 
be the double mapping cylinder of $\mathbf{X}\setminus\mathbf{Z}\hookleftarrow\mathbf{A}\setminus\mathbf{Z}\hookrightarrow\mathbf{A}$. The point-set topological conditions on $\mathbf{A}$ and $\mathbf{Z}$ imply that the collapse map $c:\mathbf{M}\rightarrow\mathbf{X}$ is an equivalence. 
Consider the following diagram, in which each row is a fiber sequence.
\[
  \begin{tikzcd}
  ^{\tau}\vec{\mathcal{K}}_{S^1}(\textbf{X},\mathbf{A})_{n} \arrow{d}  \arrow{r} &^{\tau}\vec{\mathcal{K}}_{S^1}(\textbf{X}) _{n}\arrow{r} \arrow{d} & ^{\tau}\vec{\mathcal{K}}_{S^1}(\mathbf{A})_{n}\arrow{d} \\
    ^{c^{*}\tau}\vec{\mathcal{K}}_{S^1}(\textbf{M},\mathbf{A})_{n}\arrow{r}&^{c^{*}\tau}\vec{\mathcal{K}}_{S^1}(\textbf{M})_{n} \arrow{r}&^{\tau}\vec{\mathcal{K}}_{S^1}(\mathbf{A})_{n}
  \end{tikzcd}
\]
The vertical arrow on the right is a homeomorphism and the vertical arrow in the middle is a homotopy equivalence, so the vertical arrow on the left is a weak homotopy equivalence. Let $\mathbf{N}$ be the mapping cylinder of $\mathbf{X}\setminus\mathbf{Z}\hookleftarrow\mathbf{A}\setminus\mathbf{Z}$ and $h:\mathbf{N}\rightarrow\mathbf{X}$ the obvious map. Since $(\mathbf{X}\setminus\mathbf{Z},\mathbf{A}\setminus\mathbf{Z})=(\mathbf{X}\setminus\mathbf{Z},\mathbf{A}\setminus\mathbf{Z})$ a similar argument shows that $h$ induces a weak equivalence 
$$^{\tau}\vec{\mathcal{K}}_{S^1}(\textbf{X}\setminus\mathbf{Z},\mathbf{A}\setminus\mathbf{Z})_{n}\longrightarrow\ ^{h^{*}\tau}\vec{\mathcal{K}}_{S^1}(\textbf{N},\mathbf{A}\setminus\mathbf{Z})_{n}.$$
Thus it suffices to show that the restriction map
$$r:\  ^{c^{*}\tau}\vec{\mathcal{K}}_{S^1}(\textbf{M},\mathbf{A})_{n}\longrightarrow\ ^{h^{*}\tau}\vec{\mathcal{K}}_{S^1}(\textbf{N},\mathbf{A}\setminus\mathbf{Z})_{n}$$
is a weak equivalence.\footnote{Note that the twists $h^{*}\tau$ and $r^{*}c^{*}\tau$ are canonically isomorphic.}

For a based topological sheaf $\mathcal{F}$ over a groupoid $\mathbf{V}$ with subgroupoid $\mathbf{A}\setminus\mathbf{Z}$ let $\Gamma(\mathbf{V},\mathbf{A}\setminus\mathbf{Z};\mathcal{F})$ denote the space of global sections whose restriction to $\mathbf{A}\setminus\mathbf{Z}$ is the basepoint. Recall that the relative $\vec{K}_{S^1}$-theory spectrum of a pair is by definition the homotopy fiber of a restriction map on the space of sections of a based topological sheaf. Since the relative inclusions $\mathbf{A}\hookrightarrow\mathbf{M}$ and $\mathbf{A}\setminus\mathbf{Z}\hookrightarrow\mathbf{N}$
are cofibrations\footnote{A cofibration of topological groupoids is defined in the same was as for topological spaces. Namely one asks for the homotopy extension property for maps to topological spaces.} the corresponding restriction maps are fibrations. Hence the maps\footnote{Since the pullback of the $\tau$-twisted locally universal Hilbert bundle over $\pmb{\tau}$ to any of the groupoids mentioned above is again locally universal, in a final abuse of notation I use $q\mathcal{F}^{n}$ to denote the sheaf and any of its pullbacks by maps in sight.}  

$$\Gamma(\textbf{M},\mathbf{A};q\mathcal{F}^{n})\longrightarrow\ ^{c^{*}\tau}\vec{\mathcal{K}}_{S^1}(\textbf{M},\mathbf{A})_{n}$$
$$\Gamma(\textbf{N},\mathbf{A}\setminus\mathbf{Z};q\mathcal{F}^{n})\longrightarrow\ ^{h^{*}\tau}\vec{\mathcal{K}}_{S^1}(\textbf{N},\mathbf{A}\setminus\mathbf{Z})_{n}$$
are inclusions of fibers into homotopy fibers and are thus homotopy equivalences.

Finally, the relative inclusion $(\mathbf{N},\mathbf{A}\setminus\mathbf{Z})\hookrightarrow(\mathbf{M},\mathbf{A})$ induces  induces a homeomorphism 
$$\Gamma(\textbf{M},\mathbf{A};q\mathcal{F}^{n})\longrightarrow\Gamma(\textbf{N},\mathbf{A}\setminus\mathbf{Z};q\mathcal{F}^{n}),$$
which---along with the previous two homotopy equivalences and the map of interest $\Xi$---fits into the following diagram. 
\[
  \begin{tikzcd}
 \Gamma(\textbf{M},\mathbf{A};q\mathcal{F}^{n}) \arrow{d}  \arrow{r} & \Gamma(\textbf{N},\mathbf{A}\setminus\mathbf{Z};q\mathcal{F}^{n})\arrow{d} \\
   ^{c^{*}\tau}\vec{\mathcal{K}}_{S^1}(\textbf{M},\mathbf{A})_{n} \arrow{r}{r}&\ ^{h^{*}\tau}\vec{\mathcal{K}}_{S^1}(\textbf{N},\mathbf{A}\setminus\mathbf{Z})_{n}
  \end{tikzcd}
\]
It follows that $r$ is a weak equivalence.

\end{proof}

\section{Calculations}

\begin{defn}\label{strongtopreg} An element $\tau\in H_{S^1}^{3}(G/G)$ is called a \textit{strongly topologically regular}\footnote{The adverb ``strongly" is there to distinguish this from the condition of topological regularity (c.f. [FHT3] 2.1), which does not require the bilinear form to be definite.} twist if the restriction of $\tau$ to $H^{3}(T/T)\simeq H^{1}(T)^{\otimes2}\oplus H^{3}(T)$ is concentrated in the first summand (c.f. Remark \ref{deg3twist}) and defines a symmetric, non-degenerate, definite bilinear form on $H_{1}(T)$. The twist is called \textit{positive} (\textit{negative}) if that bilinear form is positive (negative) definite.
\end{defn}

Let $G(1)$ denote the identity component of a compact Lie group $G$. In this section the groups $^{\tau}\vec{K}^{*}_{S^1}(G(1)/G)$ will be calculated. The calculation will eventually break into two subsections---according to whether the twist $\tau$ is positive or negative definite---which have different flavors. The asymmetry not mysterious: it comes precisely from the asymmetry in completing $\mathbb{Z}[q^{\pm}]$ to $\mathbb{Z}((q))$ rather than $\mathbb{Z}[[q,q^{-1}]]$. Morally, relaxing finite-dimensionality of vector bundles in \textit{positive} powers of $q$ doesn't amount to much at \textit{negative} level, since the cocycles (are expected to) correspond to loop group representations, which at negative level are of \textit{negative energy} and have bounded \textit{above} $S^1$-eigenspaces. Thus the representations themselves do not produce $\vec{K}_{S^1}$-cocycles without being `finitized' by a Fredholm operator\footnote{This is the (rather involved) FHT-Dirac construction (c.f. [FHT3] Section V).} and the story collapses to match classical twisted $K$-theory.
 
Let $T$ denote a maximal torus of $G$ and let $N$ be the normalizer of $T$. The plan is to make a preliminary calculation over $T/N$ and then transport that to $G(1)/G$ via the natural faithful map $\omega:T/N\rightarrow G/G$. A few lemmas can be stated and proved uniformly for positive and negative twists. 

\begin{lem}\label{stalksofKhat} (The stalks of $\vec{K}^*_{S^1})$ 
Let $G$ be a compact Lie group and $z\in Z(G)$ a central element. Then $z$ defines a $B\mathbb{Z}$-action on $\text{pt}/G$, and the $B\mathbb{Z}$-quotient (c.f. Definition \ref{BZquotient}) is of the form $\text{pt}/\tilde{G}$ where $G\rightarrow \tilde{G}\rightarrow S^1$ is a group extension. Any twist $\tau\in H^3_{S^1}(BG)$ defines a $U(1)$-central extension $\tilde{G}^\tau\rightarrow  \tilde{G}$. Let $G^\tau$ be the pullback of $\tilde{G}^\tau$ to $G$, so that $\tilde{G}^\tau/G^\tau=S^1$. Let $
\tilde{\Lambda}^\tau$ and $\Lambda^\tau$ be the set of $\tau$-affine (c.f. Footnote \ref{affinecharacter}) weights of $\tilde{G}^\tau$ and $G^\tau$, and let $W$ be the Weyl group\footnote{The Weyl groups of all extensions in sight are canonically isomorphic.}. Let $
\tilde{\Lambda}^\tau_\text{di}\subset\tilde{\Lambda}^\tau$ and $\Lambda^\tau_\text{di}\subset\Lambda^\tau$ be the subsets of dominant\footnote{I take this to mean dominant with respect to \textit{any} choice of positive Weyl chamber.} integral weights. Let $\tilde{\pi}:\tilde{\Lambda}^\tau_\text{di}/W\rightarrow\text{pt}/G$ and $\pi:\Lambda^\tau_\text{di}/W\rightarrow\text{pt}/G$ be the projections. Any choice of fractional splitting $\psi:\tilde{G}^\tau\leftarrow S^1_d$ (c.f. Section 2.1) induces an isomorphism $\Psi:\tilde{\Lambda}^\tau\leftarrow \Lambda^\tau\times \mathbb{Z}$. Let $K^*_c$ denote $K$-theory with compact supports\footnote{This is the reduced $K$-theory of the one-point compactification.} and let $^{\tilde{\pi}^*\tau}K^0_\star(\tilde{\Lambda}^\tau_\text{di}/W)\subset\ ^{\tilde{\pi}^*\tau}K^0(\tilde{\Lambda}_\text{di}^\tau/W)$ be the subgroup of classes whose image under $\Psi^*:\ ^{\tilde{\pi}^*\tau}K^0(\tilde{\Lambda}^\tau_\text{di}/W)\rightarrow \ ^{\pi^*\tau}K^0(\Lambda^\tau_\text{di}/W\times\mathbb{Z})=\ ^{\pi^*\tau}K^0(\Lambda^\tau_\text{di}/W)[[q,q^{-1}]]$ is contained in $^{\pi^*\tau}K^0_c(\Lambda^\tau_\text{di}/W)((q))$ for all fractional splittings $\psi$. Let $\tilde{V}\rightarrow \tilde{\Lambda}^\tau_\text{di}/W$ be the canonical vector bundle whose fiber at a point is a copy of the $\tilde{G}^\tau$-representation labelled by that point. If $M$ is and $R$-module and $r\in R$, write $rM$ for the subgroup $\{rm| m\in M\}$.  Then $^\tau\vec{K}^1_{S^1}(\text{pt}/G)=0$ and `summation along the fiber' defines an isomorphism
$$\tilde{\pi}_!:[\tilde{V}]\big(\ ^{\tilde{\pi}^*\tau}K^0_\star(\tilde{\Lambda}^\tau_\text{di}/W)\big)\rightarrow\ ^\tau\vec{K}^0_{S^1}(\text{pt}/G).$$
Finally, for $M\in R^\tau(G)$ write $\chi_M$ for its character. For $\gamma\in\text{Hom}(S^1,G)$ and $\xi=\sum M_kq^k\in R^\tau(G)((q))$ define $\gamma\cdot\xi=\sum M_k\chi_{M_k}(\gamma(q))q^k$. Then any choice of local splitting $\psi$ induces an injection $^\tau\vec{K}^0_{S^1}(\text{pt}/G)\hookrightarrow R^\tau(G)((q))$ which is an isomorphism onto the subgroup $R_\star^\tau(G)((q))$ of elements $\xi$ for which $\gamma\cdot\xi$ is in $R^\tau((q))$ for all $\gamma\in\text{Hom}(S^1,G)$.
\end{lem}

\begin{proof}
The first claim that is not immediate from the definitions and the standard representation theory of compact Lie groups is the implicit well-definedness of the displayed map. Let $R^\tau(G)$ be the group of $\tau$-projective\footnote{\label{tauprojectiverep}That is, the subgroup of the representation ring $R(G^{\tau})$ where the central $U(1)$ acts by scalar multiplication.} representations of $G$. It may be helpful to recall that the uncompleted analog of this map is the isomorphism $\tilde{\pi}_!^c:[\tilde{V}]\big(\ ^{\tilde{\pi}^*\tau}K^0_c(\tilde{\Lambda}^\tau_\text{di}/W)\big)\rightarrow\ ^\tau K^0_{S^1}(\text{pt}/G)=R^\tau(\tilde{G})$ defined as follows: for $\lambda\in \tilde{\Lambda}^\tau_\text{di}/W$ let $\tilde{V}_\lambda$ denote a copy of the corresponding irreducible representation $G^\tau$. Then a compactly supported virtual vector bundle $E\rightarrow\tilde{\Lambda}^\tau_\text{di}/W$ is sent to the direct sum $\oplus_{\tilde{\Lambda}^\tau_\text{di}/W}\tilde{V}_\lambda\otimes E_\lambda$. 

It is immediate from the definition (c.f. Definition \ref{qfred} and \ref{Khatdef}) that classes in $^\tau\vec{K}^0_{S^1}(\text{pt}/G)$ are represented by certain possibly infinite dimensional representations of $\tilde{G}^\tau$. To describe them, choose a fractional splitting $\psi:\tilde{G}^\tau\leftarrow S^1_d$. From Definition \ref{qfred} it follows that there is an injective map $i_\psi:\  ^\tau\vec{K}^0_{S^1}(\text{pt}/G)\rightarrow R^\tau(G)((q))$. For an element $M\in R^\tau(G)$ let $\chi_M$ denote its virtual character. Any other choice of fractional splitting $\psi'$ is of the form $\psi'(q)=\psi(q)\gamma(q)$ for some $\gamma:S^1\rightarrow G$. Therefore, if $\xi\in\ ^\tau\vec{K}^0_{S^1}(\text{pt}/G)$ and $i_\psi\xi=\sum M_kq^k\in R^\tau(G)((q))$ then $i_{\psi'}\xi=\sum M_k\chi_{M_k}(\gamma(q))q^k\in R^\tau(G)((q))$. So again it follows directly from Definition \ref{qfred} again that $i_\psi$ is an isomorphism onto the subgroup $R_\star^\tau(G)((q))\subset R^\tau(G)((q))$ of elements $\sum M_kq^k$ for which $\sum M_k\chi_{M_k}(\gamma(q))q^k$ is also in $R^\tau(G)((q))$.

Now the image $\pi_!([V]\xi)$ under the displayed map in the lemma (defined by the same formula that ends the fist paragraph of this proof) certainly defines a possibly infinite dimensional representation of $\tilde{G}^\tau$. By assumption $\Psi^*\xi\in\  ^{\pi^*\tau}K^0_c(\Lambda^\tau_\text{di}/W)((q))$, so write $\Psi^*\xi=\sum E_kq^k$. Moreover, for another splitting $\psi'=\psi\gamma$ (see the previous paragraph) we also have $(\Psi')^*\xi\in\  ^{\tilde\pi^*\tau}K^0_c(\Lambda^\tau_\text{di}/W)((q))$. In analogy with $\tilde{V}$ let $V\rightarrow \Lambda^\tau_\text{di}/W$ be the canonical vector bundle whose fiber at a point is a copy of the $G^\tau$ representation labeled by that point. But it is clear that $\pi_!^c([V]\Psi^*\xi)=i_\psi(\tilde{\pi}_![\tilde{V}]\xi)$, so the conditions defining $^{\tilde{\pi}^*\tau}K^0_\star(\tilde{\Lambda}^\tau_\text{di}/W)$ are tautologically the conditions for $\tilde{\pi}_!$ to be well-defined. Since $\pi_!^c$ is an isomorphism (see the first paragraph of this proof), the same equation shows that $\i_\psi\circ\tilde{\pi}_!$ is an isomorphism onto its image, which finishes the proof.
\end{proof}

\begin{lem}\label{twistfamily}
Let $N$ be the normalizer of a maximal torus $T\subset G$ and write $W=N/T$.  Let $\tau$ be a twist in $H^3_{S^1}(T/N)$. For each $t\in T$ with stabilizer $N_t\subset N$, $\tau$ defines a group $\tilde{N}_t^{\tau(t)}$ which is a  $U(1)$-central extension of an $N_t$-extension $\tilde{N}_t$ of $S^1$, and also an extension of $S^1$ by a $U(1)$-extension $N^{\tau(t)}$ of $N_t$. This is more lucidly indicated in the commutative diagram below:
$$\begin{tikzcd}
 U(1)\arrow[r,hook]&N^{\tau(t)}_t\arrow[r]\arrow[d] &N_t\arrow[d,hook] \\
 U(1)\arrow[r,hook]&\tilde{N}_t^{\tau(t)}\arrow[d]\arrow[r] &\tilde{N}_t\arrow[d] \\
 & S^1 &S^1 
\end{tikzcd}$$
Restriction to the maximal torus $T\subset N_t$ defines a group $\tilde{T}^{\tau(t)}$ which is a $U(1)$-central extensions of $T$-extensions of $S^1$, and as $t$ varies these $\tilde{T}^{\tau(t)}$ assemble into a bundle of groups over $T$ whose fiber over $1$ is equal to $U(1)\times T\times S^1$.
\end{lem}
\begin{proof} First note that the explicit model of the $B\mathbb{Z}$-quotient $T/N/B\mathbb{Z}$ given in Definition \ref{BZquotient} produces a bundle of groups $T\times N\times_\mathbb{Z}\mathbb{R}\rightarrow T$ whose fiber at $t\in T$ is an extension $N\rightarrow \tilde{N}_t\rightarrow \mathbb{R}/\mathbb{Z}=S^1$. 
For each point $t\in T$ with stabilizer $N_t\subset N$, pullback along the inclusion $i_t:\{t\}/N_t\rightarrow T/N$ produces a class $\tau(t):=i_t^*\tau\in H^3(\{t\}/N_t):=H^3BN_t$, i.e a central extension $U(1)\rightarrow N^{\tau(t)}_t\rightarrow N_t$. Pulling back along $T\rightarrow N_t$ produces the desired central extension of $T$, denoted $T^{\tau(t)}\rightarrow T$. The ability to trivialize the extensions at the fiber over $1\in T$ is a direct consequence of the fact that all extensions of $S^1$ by a torus are trivializable, and all extensions of a torus by $U(1)$ are trivializable. The bundle of groups over $T$ is presented explicitly in Definition \ref{explicitmodelTtau} below.
\end{proof}


\begin{defn}\label{explicitmodelTtau} Let $\tau\in H^3_{S^1}(T/T)$ be a strongly topologically regular twist and let $\tau$ also denote the corresponding bilinear form on $H_1T=\pi_1T$ (c.f. Definition \ref{strongtopreg}). Let $\Lambda$ be the character lattice of $T$. Contraction with the bilinear form on $H^{1}T$ defined by $\tau$ gives a map $\kappa^{\tau}:H_{1}T\rightarrow H^{1}T=\Lambda$. Define an action of $\pi_1T$ on $\mathfrak{t}\times U(1)\times T\times\mathbb{R}$ by the formula 
$$p\cdot(X,e^{2\pi i\theta},t,r)=(X+p,e^{2\pi i(\theta+\tau(p,p)r)}\kappa^\tau p(e^X)^{-1},tp(e^{2\pi ir}),r).$$
There is a commuting $\mathbb{Z}$-action (c.f. Definition \ref{BZquotient}) defined by the formula 
$n\cdot(X,e^{2\pi i\theta},t,r)=(X,e^{2\pi i\theta},te^{-nX},r+n)$
and write $L^\tau$ for the quotient by $\pi_1T\times\mathbb{Z}$. Define the groupoid $\mathbf{L}^\tau=(L^\tau\rightrightarrows T)$ by declaring the source and target map to both be $[(X,\theta,t,r)]\mapsto e^X$. This defines the bundle of groups over $T$ alluded to in Lemma \ref{twistfamily}. The map $[(X,e^{2\pi i\theta},t,r)]\mapsto[(X,t,r)]$ defines a morphism $\mathbf{L}^\tau\rightarrow T/T/B\mathbb{Z}$ which presents the $B\mathbb{Z}$-equivariant twist associated to $\tau$. 
\end{defn}
 
\begin{lem}\label{ptau}
There is a $W$-equivariant covering space $\pi:P_\tau\rightarrow T$ which is the bundle of affine weights\footnote{\label{affinecharacter}
Recall that an \textit{affine} weight of a central extension $U(1)\rightarrow G\rightarrow H$ is a weight of $G$ which restricts to the identity character of $U(1)$, which is contained in any maximal torus.} associated to the bundle of central extensions defined by the groupoid $\mathbf{L}^\tau$ of  Definition \ref{explicitmodelTtau}. In particular, for each $t\in T$, if $T^{\tau(t)}$ denotes the automorphism group of $t$ in the groupoid $\mathbf{L}^\tau$ and $\tilde{\Lambda}^{\tau(t)}$ is the set of affine weights of $T^{\tau(t)}$, then there is a canonical isomorphism $\pi^{-1}(t)\simeq\tilde{\Lambda}^{\tau(t)}.$
\end{lem}

\begin{proof} 
The proof is an explicit construction of $P_\tau$. Define an action of  $\pi_{1}T$ on $\mathfrak{t}\times\Lambda\times\mathbb{Z}$ by (c.f. [PS] 4.9.5)
$$\pi_{1}T\ni p:(X,\lambda,n)\mapsto(X+p,\lambda-\kappa^{\tau}p,n+\tau(\kappa^{\tau}p,\kappa^{\tau}p) + \lambda(p)).$$
The desired covering map is
$$\pi:P_\tau:=\mathfrak{t}\times_{\pi_{1}T}(\Lambda\times\mathbb{Z}) \rightarrow T$$
$$[(X,\lambda,n))]\mapsto [X].$$ 
To identify $\pi^{-1}(t)$ with $\tilde{\Lambda}^{\tau(t)}$, view $L^\tau$ as the $\pi_1T$-quotient of $(\mathfrak{t}\times U(1)\times T\times\mathbb{R})/\mathbb{Z}$. For $(t,r)\in T\times \mathbb{R}$ write $[r]$ for the corresponding element of $S^1=\mathbb{R}/\mathbb{Z}$, and $t_r(X)=te^{-n_rX}$ where $n_r\in \mathbb{Z}$ is such that $r+n_r\in[0,1)$. Then $(\mathfrak{t}\times U(1)\times U(1) T\times\mathbb{R})/\mathbb{Z}$ can be identified with $\mathfrak{t}\times U(1)\times T\times S^1$ via the map $[(X,e^{2\pi i\theta},t,r)]\mapsto (x,e^{2\pi i\theta},t_0,[r])$, and under this identification the $\pi_1T$ action on $\mathfrak{t}\times U(1)\times T\times S^1$ becomes $p\cdot(X,e^{2\pi i\theta},t,\phi)=(X+p,e^{2\pi i\theta+\tau(p,p)\phi}\kappa^\tau p(e^X)^{-1},tp(e^{2\pi i\phi}),\phi).$ The associated action on the subset of $\mathfrak{t}\times\text{Hom}(U(1)\times T\times S^1,U(1))$ consisting of pairs $(X,f)$ such that $f$ restricts to the identity character of the $U(1)$ factor is precisely the action defining $P_\tau$.
\end{proof}


\begin{defn}\label{compactification}
Define a partial compactification $\vec{P}_\tau$ of $P_\tau$.  
Let $T_\infty$ denote a copy of $T$ with the trivial $W$-action. As a set, the partial compactification of $P_{\tau}$ is $\vec{P}_{\tau}=P_{\tau}\coprod T_\infty$. Let $\vec{\pi}:=\pi\coprod \text{id}_{T}:\vec{P}_{\tau}\rightarrow T$ be the natural projection. The topology of $\vec{P}_{\tau}$ is generated by the open sets of $P_{\tau}$, the sets $\vec{\pi}^{-1}(U)$ for $U\subset T$ open, and the collection of sets
 $\Big\{P_{\tau}\setminus C\coprod T_\infty\Big\}$ such that
\begin{enumerate}[itemsep=0pt]
\item (a `niceness' condition) $C\subset P_\tau$ is closed and its preimage in $\mathfrak{t}\times\Lambda\times\mathbb{Z}$ has convex intersection with $\mathfrak{t}\times\{\lambda\}\times\{n\}$ for all $\lambda$ and $n$ and
\item (a condition directly related to the definition of $\vec{K}_{S^1}^*$, c.f Lemma \ref{stalksofKhat}) for any $t\in T$ and any lift $X\in \mathfrak{t}$ (i.e. $e^X=t$), the preimage of $C\cap \pi^{-1}(t)$ under the isomorphism $\pi^{-1}(t)\xleftarrow{\sim} \Lambda^{\tau(t)}\times\mathbb{Z}$ defined by $X$ is a subset whose intersection with $\Lambda\times {n}$ is finite for all $n$ and empty for sufficiently negative $n$.
\end{enumerate}
\end{defn}


 \begin{lem}\label{keylemma} (Key Lemma\footnote{This is the analog of the `Key Lemma' of [FHT3] (Lemma 5.2).}) Let $N\subset G$ be the normalizer of $T$ and $W=N/T$ the Weyl group. Note that $T/N$ is a full $B\mathbb{Z}$-subgroupoid of $\mathcal{L}(\text{pt}/N)\simeq N/N$.  Let $\tau$ be a strongly topologically regular (c.f. Definition \ref{strongtopreg}) twist in $H^3_{S^1}(T/N)$. There is a $W$-equivariant map $
 \vec{\pi}:\vec{P}_\tau\rightarrow T$, a subspace $T_\infty\subset\vec{P}_\tau$ and an isomorphism of $\mathbb{Z}((q))$-modules
$$\alpha:\ ^{\vec{\pi}^*\tau} K^{*}(\vec{P}_{\tau}/W, T_\infty/W)\xrightarrow {\sim}\ ^{\tau} \vec{K}^{*}_{S^1}(T/N).$$

\end{lem}
\begin{rem}
Note that the domain of the display is an ordinary twisted $K$-theory (not $\vec{K}_{S^1}$!) group. Its $\mathbb{Z}((q))$-module structure is described in the proof.
\end{rem}

\begin{proof} 

To prove the lemma the first order of business is to define a $\mathbb{Z}((q))$-module structure on the domain of the displayed map. To do that it suffices to specify the action of $q$, which is defined to act via the map on twisted $K$-theory induced by pullback along the shift isomorphism `sh' defined by $\text{sh}([(X,\lambda,n)])=[(X,\lambda,n-1)]$.

The next order of business is to define the map $\alpha$. Recall that $P_\tau$ is the bundle of (affine) characters corresponding to a bundle of groups over $T$ (c.f. Lemma \ref{ptau}). In particular, any point in $s\in P_\tau$ defines a 1-dimensional representation: namely if $\pi(s)=t$ then the character labeled by $s$ defines a representation of the group $T^{\tau(t)}$ (c.f. Lemma \ref{twistfamily}) on $\mathbb{C}$, and these assemble into a vector bundle $V\rightarrow P_\tau$. Extending by a trivial rank 1 vector bundle over $T_\infty$ defines a vector bundle $\vec{V}\rightarrow\vec{P}_\tau$. The \textit{proposed} definition of $\alpha$ is the composite
$$^{\vec{\pi}^*\tau} K^{*}(\vec{P}_{\tau}/W, T_\infty/W)\xrightarrow{\otimes[\vec{V}]}\ ^{\vec{\pi}^*\tau} K^{*}(\vec{P}_{\tau}/W, T_\infty/W) \xrightarrow{\vec{\pi}_!}\ ^{\tau} \vec{K}^{*}_{S^1}(T/N),$$
but some proof is required to show that this is well defined. Namely we need to show that for any $\xi$ in the image of the first map, the possibly infinite dimensional vector bundle $\vec{\pi}_!\xi$ given by `summation along the fibers' satisfies the conditions to define a class in the target. Since $\vec{\pi}$ is a fiber bundle over an equivariantly locally contractible base, it suffices to check this pointwise in the base. So choose a point $t\in T$, with stabilizer $N_{t}$ and `local Weyl group' $W_t=N_t/T$. Write $\vec{V}_t$ for the restriction of $\vec{V}$ to $\vec{\pi}^{-1}(t)$. We must show that `summation along the fibers' produces a well-defined map
$$(\vec{\pi}_t)_!:[\vec{V}_t]\Big(\ ^{\vec{\pi}^*\tau(t)} K^{*}(\vec{\pi}^{-1}(t)/W_t, \{t\}_\infty/W_t)\Big)\rightarrow\ ^{\tau(t)} \vec{K}^{*}_{S^1}(\{t\}/N_t).$$
Let $\tilde{\Lambda}^{\tau(t)}$ and $\Lambda^{\tau(t)}$ denote the set of $\tau(t)$-affine weights of $\tilde{N}_t$ and $N_t$. Note that $\vec{\pi}^{-1}(t)=\tilde{\Lambda}^{\tau(t)}\coprod\{t\}_\infty$ (c.f. Lemma \ref{compactification}), and any choice of fractional splitting $\psi_t:\tilde{N}_t\leftarrow S^1_d$ induces an isomorphism $\Psi_t:\tilde{\Lambda}^{\tau(t)}\leftarrow\Lambda^{\tau(t)}\times\mathbb{Z}$. Given the results of Lemma \ref{stalksofKhat} (note that $\tilde{\Lambda}^{\tau(t)}=\tilde{\Lambda}^{\tau(t)}_\text{di}$ and $\Lambda^{\tau(t)}=\Lambda^{\tau(t)}_\text{di}$), the map is zero in degree 1 and so it suffices to show that the degree 0 part of the domain of the previously displayed map is isomorphic to $[\vec{V}_t]\big(\ ^{\vec{\pi}^*\tau(t)}K^0_\star(\tilde{\Lambda}^{\tau(t)}/W_t)\big)$ as a subgroup of $^{\vec{\pi}^*\tau(t)}K^0(\tilde{\Lambda}^{\tau(t)}/W_t)$. But this is straightforward (the topology pf $\vec{P}_\tau$ was concocted for this purpose): any relative class in the domain of the previously displayed may must have support contained in a set $C$ satisfying the two condition listed in the definition of the topology of $\vec{P}_\tau$. I claim that this support condition on classes in $^{\vec{\pi}^*\tau(t)}K^0(\tilde{\Lambda}^{\tau(t)}/W_t)$ is equivalent the definition of $^{\vec{\pi}^*\tau(t)}K^0_\star(\tilde{\Lambda}^{\tau(t)}/W_t)$. 
This follows immediately from a few observations: by Lemma \ref{ptau} there is a canonical isomorphism $\pi^{-1}(t)=\tilde{\Lambda}^{\tau(t)}$. Thus, the set of induced isomorphisms $\Psi:\tilde{\Lambda}^{\tau(t)}\leftarrow \Lambda^{\tau(t)}$ induced by choices of fractional splitting $\psi:\tilde{N}_t^{\tau(t)}\leftarrow S^1_d$ is equal to the set of isomorphisms $\pi^{-1}(t)\rightarrow  \Lambda^{\tau(t)}$ induced by choices of lifts $X\in\mathfrak{t}$, $e^X=t$.

Not only have I shown that $\alpha$ is well-defined, but also that it is an isomorphism at each point. By equivariant local contractibility of $T/N$ it follows that $\alpha$ is locally an isomorphism, and by the excision axiom (or equivalently, the Mayer-Vietoris axiom) it follows that $\alpha$ is an isomorphism. 

It remains to show that $\alpha$ is a $\mathbb{Z}((q))$-module map. That follows immediately form the definition of the $q$-action in the first paragraph of this proof together with the fact that the `shift' isomorphism defined there coincides on each fiber $\pi^{-1}(t)=\tilde{\Lambda}^{\tau(t)}$ with the action of the generator of $\text{Hom}(S^1,U(1))$ (recall that $\tilde{\Lambda}^{\tau(t)}$ is a subgroup of $\text{Hom}(T^{\tau(t)},U(1))$). 
\end{proof}

Consider the natural inclusion $\omega:T/N\rightarrow G/G$.
Recall that the $B\mathbb{Z}$-actions on both groupoids are the automorphisms of the identity defined by $t\mapsto t$ and $g\mapsto g$ (c.f. Example \ref{loopgroupoid}). Hence the map is $B\mathbb{Z}$-equivariant (c.f. Definition \ref{BZdef}) and so there is a corresponding map in $\vec{K}_{S^1}$-theory 
$$ ^{\tau}\vec{K}^{*}_{S^1}(G/G) \xrightarrow{\omega^{*}}\ ^{\omega^*\tau} \vec{K}^{*}_{S^1}(T/N).$$
I would like to define a pushforward. 
Let $N$ act on $G\times T$ and $G\times G$ by the formula $n(g,k)=(gn^{-1},nkn^{-1})$. Let $G$ act on the quotients $G\times_NT$ and $G\times_NG$ by left translation on the left factor. 
The natural inclusion $G\times_NT\hookrightarrow G\times_NG$ induces a fully faithful map $i:(G\times_NT)/G\hookrightarrow(G\times_NG)/G$. Moreover, the inclusion $N\hookrightarrow G$ and the natural map $T=\{1\}\times T\rightarrow G\times_NT$ define an equivlanece $T/N\rightarrow(G\times_NT)/G$. Finally, the map $G\times_NG\rightarrow G$ defined by $[(g,k)]\mapsto gkg^{-1}$ defines a map $j:(G\times_NG)/G\rightarrow G/G$.

\begin{defn}\label{segalind} Let $\iota:N\rightarrow G$ be a map of compact Lie groups whose kernel is finite and whose image is a closed Lie subgroup. Define \textit{Segal induction} (c.f. [Se] Section 2) as follows: first suppose that $\iota$ is injective. Then Segal induction is the map $\iota_!:R(N)\rightarrow R(G)$ which sends $M\in R(N)$ to the analytic index of $d+d^{*}$ acting on the de Rham complex of the associated virtual vector bundle $G\times_{N} M\rightarrow G/N$, in a chosen equivariant orthogonal structure. If $\iota$ has finite kernel $K$, define Segal induction to be the composite of the `take $K$-invariants' map $R(N)\rightarrow R(N/K)$ and the previously defined Segal induction along the injective map $N/K\rightarrow G$.
\end{defn}
\begin{lem}\label{segalcharacter}
Let $\iota:N\hookrightarrow G$ be the inclusion of a closed subgroup into a compact Lie group. Let $\iota_!:R(N)\rightarrow R(G)$ denote Segal induction (c.f. Definition \ref{segalind}). If $M$ is a virtual representation let $\chi_M$ denote its character. For a regular element $t\in G$ let $F_t$ denote the set of cosets $gN\in G/N$ such that $g^{-1}tg\in N$. Then $F_t$ is finite and for any $M\in R(N)$
$$\chi_{\iota_!M}(t)=\sum_{gN\in F_t}\chi_M(g^{-1}tg).$$
Since regular elements are dense in $G$ this determines $\chi_{\iota_!M}$ completely. Finally, suppose that $N$ is the normalizer of a maximal torus $T\subset G$. Then for regular $t\in T$ the set $F_t$ is the singleton $\{1N\}$, so $\chi_{\iota_!M}$ and $\chi_M$ agree on $T$.
\end{lem}
\begin{proof} Note that $g^{-1}tg \in N$ is equivalent to $tgN=gN$, i.e. the condition for $gN\in G/N$ to be a fixed point of the action of $t\in G$. The finiteness of $F_t$ is then [Se] Proposition 1.9. The character formula is a direct consequence of the Atiyah-Bott fixed point formula (c.f. [Se], end of Section 2). To prove the alst statement, note that if $t\in T$, then $g^{-1}tg\in N$ implies that $g^{-1}tg=t'\in T$, and for any $s\in T$ $g^{-1}sg\in Z(t')_1$ (the identity component of the centralizer of $t'$). So if $t$ is regular, $g^{-1}Tg\subset T$, so $g\in N$. Since regular elements are dense in $T$ the character formula gives the desired equality of characters on $T$.

\end{proof}

\begin{lem}\label{omegafactorsegal}
The map $\omega:T/N\rightarrow G/G$ factors into two maps that admit pushforwards in classical twisted $K$-theory
$$\begin{tikzcd}
 T/N \arrow[r,"\sim"]& (G\times_NT)/G\arrow[r,hook,"i"] &(G\times_NG)/G\arrow[r,"j"] & G/G.
\end{tikzcd}$$
For any twist $\tau\in H^3(G/G)$ and any point $t\in T$ with stabilizers $N_t$ and $G_t$ in $N$ and $G$, the composite pushforward $j_*i_*$ coincides with Segal induction ($\tau(t)$ is defined in the proof):
$$j_*i_*:\ ^{\tau(t)}K^*(\{t\}/N_t)\simeq R^{\tau(t)}(N_t)\rightarrow R^{\tau(t)}(G_t)\simeq\ ^{\tau(t)}K^*(\{t\}/N_t).$$
\end{lem}
\begin{warn}\label{BZwarn}
This is not a $B\mathbb{Z}$-equivariant factorization. The groupoid $G\times_NG/G$ admits no obvious $B\mathbb{Z}$-action for which $j$ is $B\mathbb{Z}$ equivariant.
\end{warn}
\begin{proof} The map $i$ is an embedding and so it has a normal bundle. The map $j$ is a fiber bundle with fiber $G/N$ and so it has a relative normal bundle. Since every vector bundle admits a (possibly twisted) Thom isomorphism in twisted $K$-theory (c.f. [FHT1] 3.6), both $i$ and $j$ admit pushforwards $i_*$ and $j_*$. It remains to identify the local behavior of the pushforward with Segal induction.

Consider a point $t\in T$ with stabilizer $G_t\subset G$. Write $T_t$ for a maximal torus of $G_t$ and let $N_t\subset N$ be the normalizer of $T_t$ in $G_t$. Write $\tau(t)$ for the restriction of $\tau$ to $\{t\}/G_t$.
A local (i.e. in an infinitesimal neighborhood of $\{t\}$) presentation of $\omega$ is the $G_t$-equivariant composite
$$\begin{tikzcd}
 \omega_t:G_t\times_{N_{t}}\mathfrak{t}_t\arrow[r,"i_t",hook]&G_t\times_{N_{t}}\mathfrak{g}_t\arrow[r,"j_t"] & \mathfrak{g}_t 
\end{tikzcd}$$ 
defined by $[(g,X)]\mapsto[(g,X)]\mapsto\text{Ad}_{g}(X)$.
 Write $\sigma(\nu_{i_t})$ for the twist of the Thom isomorphism along the normal bundle $\nu_{i_t}$ of $i_t$, which is the vector bundle $\nu_{i_t}:G_t\times_{N_t}(\mathfrak{t}_t\oplus\mathfrak{g}_t/\mathfrak{t}_t)\rightarrow G_t\times_{N_t}\mathfrak{t}_t$. Linearly  contracting $\mathfrak{t}_t$ and $\mathfrak{g}_t$ (which is $G_t$-equivariant) induces the vertical isomorphisms in the following diagram (to ease the notational burden I have left various twists syntactically unspecified, they will not be referred to again) 
 $$\begin{tikzcd}
  & ^\sigma\vec{K}_{G_t}^*(G_t\times_{N_t}\mathfrak{t}_t)\arrow[r,"(i_t)_*"]
  &^{\sigma'}\vec{K}_{G_t}^*(G_t\times_{N_t}\mathfrak{g}_t)\arrow[r,"(j_t)_*"] &\ ^{\sigma''}\vec{K}_{G_t}^*(\mathfrak{g}_t) & \\
   ^{\tau(t)}K_{N_t}(\{t\})\arrow[r,"\sim"] & ^{\sigma'''}\vec{K}_{G_t}^*(G_t/N_t))\arrow[u,"\sim"vert]\arrow[r,"(i_t')_*"]&^{\sigma^{(iv)}}\vec{K}_{G_t}^*(G_t/N_t)\arrow[r,"(j_t')_*"]\arrow[u,"\sim"vert] &^{\sigma^{(v)}}\vec{K}_{G_t}^*(\text{pt})\arrow[u,"\sim"vert] \arrow[r,"\sim"]& ^{\tau(t)}K_{G_t}(\{t\}).
 \end{tikzcd}$$
Since the inverse of the equivariant contraction of $\mathfrak{t}_t$ is the inclusion $G_t\times_{N_t}\{0\}\hookrightarrow G_t\times_{N_t}\mathfrak{t}_t$ the vector bundle $\nu_{i_t}$ restricts to the vector bundle $G_t\times_{N_t}\mathfrak{g}_t/\mathfrak{t}_t\rightarrow G_t/N_t$, which is the tangent bundle $T(G_t/N_t)$. Hence the horizontal map $(i_t')_*$ is multiplication by the euler class\footnote{The euler class depends on a choice of complex orientation of $K$-theory, I am not specifying one because shortly it will not matter.} of $T(G_t/N_t)$. 

Now $(j_t')_*$ is the pushforward along the $G_t$-equivariant map $G_t/N_t\rightarrow\text{pt}$. So the composite of of the bottom row is ``multiply by the euler class of $G_t/N_t$ and pushforward along $G_t/N_t\rightarrow\text{pt}$." This coincides with the Becker-Gottlieb transfer along $G_t/N_t\rightarrow\text{pt}$\footnote{This follows directly from the definitions of the Becker-Gottlieb transfer and of the pushforward, c.f. [N] for a detailed construction of the Becker-Gottlieb transfer in equivariant cohomology theories.}. This extends to the twisted setting: if $G^\tau\rightarrow G$ is a central extension of $G$ inducing a central extension $H^{\iota^*\tau}\rightarrow H$ of $H$ then Segal induction along $\iota^\tau:H^{\iota^*\tau}\hookrightarrow G^\tau$ produces a map $\iota^\tau_!:R(H^{\iota^*\tau})\rightarrow R(G^\tau)$ which I claim restricts to a map $R^{\iota^*\tau}(H)\rightarrow R^\tau(G)$ (c.f. Footnote \ref{tauprojectiverep}). Indeed this follows from the last statement in Lemma \ref{segalcharacter} since $U(1)$ is by definition central in $H^{\iota^*\tau}$ and $G^\tau$ and hence contained in any maximal torus. 
\end{proof}
\begin{lem}\label{omegapush}
The inclusion $\omega:T/N\rightarrow G/G$ admits a pushforward in $\vec{K}_{S^1}^*$-theory.
\end{lem}
\begin{proof}
The proof is an explicit construction of the map, leveraging the existence of the non-$B\mathbb{Z}$-equivariant pushforward in classical twisted $K$-theory provided by Lemma \ref{omegafactorsegal}. The plan is to define $\omega_*$ on very small open sets and prove that these patch together to a globally defined map.
Fix $t\in T$ with $N$-stabilizer $N_t$ and $G$-stabilizer $G_t$. Fix a sufficiently small neighborhood $i_U:U\hookrightarrow G$ such that $U$ and $\omega^{-1}(U)=U\cap T$ are locally contractible, so that $U/G\simeq \{t\}/G_T$ and $\omega^{-1}(U)/N\simeq \{t\}/N_t$. Every choice of fractional splitting $\psi_N$ of
$N_t\rightarrow\tilde{N}_t\rightarrow S^1$ (c.f. Lemma \ref{stalksofKhat}) induces a fractional splitting $\psi_G$ of $G_t\rightarrow\tilde{G}_t\rightarrow S^1$. By Lemma \ref{stalksofKhat} these splittings provide an identification of $^{i_U^*\omega^*\tau}\vec{K}_{S^1}^*(\omega^{-1}(U)/N)$ with a subgroup $R_\star^{\omega^*\tau(t)}(N_t)((q))$ of  $^{i_U^*\omega^*\tau}K^*(\omega^{-1}(U)/N)((q))$ and an identification of  $^{i_U^*\tau}\vec{K}_{S^1}^*(U/G)$ with a subgroup  $R_\star^{\tau(t)}(G_t)((q))$ of $^{i_U^*\tau}K^*(U/G)((q))$.
Applying the non-equivariant pushforward $j_*i_*$ power-by-power in $q$ defines the right most vertical map in the following diagram, whose dashed arrows indicate maps that would make the diagram commute but are yet to be proven well-defined

$$\begin{tikzcd}
 ^{i_U^*\omega^*\tau}\vec{K}_{S^1}^*(\omega^{-1}(U)/N)\arrow[d,dashed,"\omega_*(U)"]\arrow[r,"\psi_N","\sim"']&R_\star^{\omega^*\tau(t)}(N_t)((q))\arrow[r,hook] \arrow[d,"\omega_*^\star(U)",dashed]& ^{i_U^*\omega^*\tau}K^*(\omega^{-1}(U)/N)((q))=R^{\omega^*\tau(t)}(N_t)((q))\arrow[d,"j_*i_*(U)"]  \\
  ^{i_U^*\tau}\vec{K}_{S^1}^*(U/G)\arrow[r,"\psi_G","\sim"']& R_\star^{\tau(t)}(G_t)((q))\arrow[r,hook]& \ ^{i_U^*\tau}K^*(U/G)((q))=R^{\tau(t)}(G_t)((q)).
\end{tikzcd}$$
To show that $\omega^\star_*(U)$ is well defined I must show that if $\xi\in R_\star^{\omega^*\tau(t)}(N_t)((q))$ then $j_*i_*(U)\xi\in R_\star^{\tau(t)}(G_t)((q))$. By lemma \ref{omegafactorsegal}, $j_*i_*(U)$ is, power-by-power in $q$, Segal induction. So an element $\xi=\sum_kM_kq^k$ is sent to $j_*i_*(U)\xi=\sum_k\iota_!M_k q^k$. By Lemma \ref{stalksofKhat}, $\xi$ defines an element of $R_\star^{\omega^*\tau(t)}(N_t)((q))$ if and only if $\gamma_N\cdot\xi=\sum_kM_k\chi_{M_k}(\gamma_N(q))q^k\in R^{\omega^*\tau(t)}(N_t)((q))$ for every $\gamma_N\in\text{Hom}(S^1,N_t)$. Now consider $\gamma_G\cdot (\sum_k\iota_!M_k q^k)=\sum_k\iota_!M_k\chi_{\iota_!M_k}(\gamma_G(q))q^k$ for $\gamma_G\in\text{Hom}(S^1,G_t)$. By the last statement of Lemma \ref{segalcharacter}, $\chi_{\iota_!M_k}(\gamma_G(q))=\chi_{M_k}(\gamma_G(q))$ because $\gamma_G(q)$ is conjugate to the maximal torus $T\subset G_t$. Hence $\omega^\star_* (U)$ is well-defined.

Now $\omega_*(U):=\psi_G^{-1}\omega^\star_*(U)\psi_N$ is well-defined and I claim it is actually independent of $\psi_N$ and $\psi_G$, as the notation suggests. Indeed, any other choice of fractional splitting of $\tilde{N}_t\rightarrow S^1$ is of the form $\psi'_N(q)=\psi(q)\gamma_N(q)$ for some $\gamma_N\in\text{Hom}(S^1,N_t)$, and the induced splitting $\psi'_G$ satisfies the same formula where $\gamma_N$ is replaced by $i_t\gamma_N$, its composite with the inclusion $i_t:N_t\rightarrow G_t$.
The claim follows by another application of the the formula $\chi_{\iota_!M_k}(\gamma_G(q))=\chi_{M_k}(\gamma_G(q))$, which implies that $\omega^\star_*(U)(\gamma_N\cdot\xi)=i_t\gamma_N\cdot\omega^\star_*(U)(\xi)$. 

It remains to show that the $\omega_*(U)$ patch together into a globally defined map, i.e. that it commutes with the restriction maps in $\vec{K}_{S^1}^*$-theory. Since locally contractible neighborhoods form a basis for the topology of $T/N$, $G/G$, their central extensions, and their $B\mathbb{Z}$-quotients, and all these groupoids have compact spaces of objects, it suffices to show that for an inclusion $j:V\hookrightarrow U$ of sufficiently small equivariantly locally contractible open neighborhoods the following diagram commutes
$$\begin{tikzcd}
    ^{\omega^*i_U^*\tau}\vec{K}_{S^1}^*(\omega^{-1}(U)/N)\arrow[r,"j^*"]\arrow[d,"\omega_*(U)"]&^{\omega^*i_V^*\tau}\vec{K}_{S^1}^*(\omega^{-1}(V)/N)\arrow[d,"\omega_*(V)"] \\
   R_\star^{\omega^*\tau(t)}(G_t)((q))=\ ^{i_U^*\tau}\vec{K}_{S^1}^*(U/G)\arrow[r,"j^*"]& ^{i_U^*\tau}\vec{K}_{S^1}^*(V/G)
\end{tikzcd}.$$
This is true power-by-power in $q$, since $j_*i_*$ is a globally defined map.
\end{proof}

\begin{lem}\label{segalkernel}
Let $G$ be a compact Lie group, $T\subset G$ a maximal torus with normalizer $N$. 
Suppose $W=N/T$ contains an Weyl reflection $r_\alpha$ defined by a root $\alpha$ of $G$. Then for any $M\in R(N)$ of virtual dimension 0 is in the kernel of the Segal induction map $\iota_!:R(N)\rightarrow R(G)$ (c.f. Definition \ref{segalind}).
\end{lem}
\begin{proof} Recall that the normalizer of the standard maximal torus of $SU(2)$ is Pin$_-(2)$. Recall that the root $\alpha$ induces maps $\delta_\alpha:\text{Pin}_-(2)\rightarrow N$ and $\Delta_\alpha: SU(2)\rightarrow G$ with finite kernel. Hence, if $r$ denotes the rank of $G$, there is a commutative diagram
$$\begin{tikzcd}
 T^{r-1}\times \text{Pin}_-(2)\arrow[r,"\delta_\alpha"]\arrow[d,hook,"\iota_\alpha"] &N\arrow[d,hook,"\iota"] \\
  T^{r-1}\times SU(2)\arrow[r,"\Delta_\alpha"]& G
\end{tikzcd}$$
Write $\epsilon\in R(\text{Pin}_-(2))$ for the sign representation $\epsilon$ of $C_2$ pulled back along the quotient $\text{Pin}_-(2)\rightarrow C_2$.
Since $M$ is in the augmentation ideal of $R(N)$ by hypothesis, and $[1-\epsilon]$ generates the augmentation ideal of $T^{r-1}\times \text{Pin}_-(2)$, $\delta_\alpha^*M$ is of the form $M'\otimes[1-\epsilon]$. I claim that such an element is sent to zero under Segal induction along $\iota_\alpha$. Indeed, it suffices to show that for any $V\in R(\text{Pin}_-(2))$, $V\otimes[1-\epsilon]$ is sent to zero by Segal induction along $\text{Pin}_-(2)\hookrightarrow SU(2)$. This follows immediately from the last sentence of Lemma \ref{segalcharacter} since $V\otimes[1-\epsilon]$ has trivial $T$-character and $SU(2)$ is connected. 

Now by [N] Lemma 4.3, $(\delta_\alpha)_!\delta_\alpha^*$ is multiplication by the $N$-equivariant Euler characteristic of the finite set coker$\delta_\alpha=W/\langle r_\alpha\rangle$. 
Since Segal induction is transitive, $(\Delta_\alpha)_!(\iota_\alpha)_!\delta_\alpha^*M=\iota_!(\delta_\alpha)_!\delta_\alpha^*M$. The left hand side is zero, and the right hand side is $|W/\langle r_\alpha\rangle|\iota_!M$. Since $|W/\langle r_\alpha\rangle|$ is nonzero and $R(G)$ is torsion free, $\iota_!M=0$.

\end{proof}

\subsection{Calculation of $^{\tau}\vec{K}^{*}_{S^1}(G/G)$ at Negative Twist}

\begin{lem}\label{abelianizednegative} Let $\tau\in H^3_{S^1}(T/N)$ be a negative twist. Equip $(\Lambda/\pi_{1}T)/W$ with the trivial $B\mathbb{Z}$-action (c.f. Definition \ref{BZdef}). There is a twist $\tau'\in H^3_{S^1}((\Lambda/\pi_{1}T)/W)$ (described in the proof) and an isomorphism of $\mathbb{Z}((q))$-modules
$$^{\tau} \vec{K}_{S^1}^{*}(T/N)\xrightarrow{\sim}\ ^{\tau'}\vec{K}_{S^1}^{*}((\Lambda/\pi_{1}T)/W).$$

\end{lem}

\begin{proof}
The plan is to invoke Lemma \ref{keylemma}, recognize the homotopy type of $\vec{P}_\tau$ as a Thom space, and apply the Thom isomorphism. 

I claim that because $\tau$ is negative, if the complement of a set $C\subset P_\tau$ contributes to the topology of $\vec{P}_{\tau}$ (c.f. Definition \ref{compactification}) then the intersection of $C$ with any connected component of $P_\tau$\footnote{Recall that these connected components are all homeomorphic to $\mathfrak{t}$ (c.f. Lemma \ref{ptau}).} has compact connected components. Indeed, by Condition 1 on the sets $C$ listed in Definition \ref{compactification} if such a component were not compact, its image in $T$ would contain a point $t$ such that $C\cap\pi^{-1}(t)$ contains points of the form $\{([X-p,\lambda,n)]\}$ for an infinite set of $p\in\pi_1T$ and some fixed $\lambda$ and $n$. But now consider the isomorphism $\pi^{-1}(t)\leftarrow\Lambda^{\tau(t)}\times\mathbb{Z}$ induced by the lift $X$ of $t$. The inverse of that map sends $[(X-p,\lambda,n)]$ to $(\lambda-\kappa^\tau p,n+\tau(\kappa^\tau p,\kappa^\tau p)+\lambda(p))$. Because $\tau$ is negative and there are infinitely many choices of $p$, that set certainly fails (the last part of) Condition 2 listed in Definition \ref{compactification}.

That leads to the following: define an action of $\pi_1T$ on $\mathfrak{t}\times\Lambda$ by $p(X,\lambda)=(X+p,\lambda-\kappa^\tau p)$. Let $(-)_+$ denote the one-point compactification and consider the spaces $\big((\mathfrak{t}\times_{\pi_{1}T}\Lambda)\times\mathbb{Z}_{\leq0}\big)_{+}$  and $\bigvee_{\mathbb{Z}_{>0}}\big(\mathfrak{t}\times_{\pi_{1}T}\Lambda\big)_{+}$. Fix and $l\in\Lambda$. For any $(X,\lambda)\in\mathfrak{t}\times\Lambda$ let $(X_l,\lambda_{l}$ be the $\pi_{1}T$-translate such that $\lambda_l$ that lies the fundamental domain containing $l$. Then the pointed map 
$$\Phi_l: \big((\mathfrak{t}\times_{\pi_{1}T}\Lambda)\times\mathbb{Z}_{\leq0}\big)_{+}\vee\bigvee_{n\in\mathbb{Z}_{>0}}\big((\mathfrak{t}\times_{\pi_{1}T}\Lambda)\times\{n\}\big)_{+}\rightarrow\vec{P}_{\tau}/T_\infty$$
$$([X,\lambda], n)\mapsto [(X_l, \lambda_{l}, n)]$$
is a $W$-equivariant homeomorphism. Hence the group $^{\vec{\pi}^*\tau}K^{*}(\vec{P}_{\tau}/W,T_\infty/W)$ can be computed as the $\Phi_l^*\vec{\pi}^*\tau$-twisted, reduced $W$-equivariant $K$-theory of the domain of $\Phi_l$. As an abelian group, that can be written suggestively in terms of $^{\Phi_l^*\vec{\pi}^*\tau}\widetilde{K}^{*}\big((\mathfrak{t}\times_{\pi_{1}T}\Lambda)_+/W\big)$ and a bookkeeping parameter presciently called $q$: 
$$^{\vec{\pi}^*\tau}K^{*}(\vec{P}_{\tau}/W,T_\infty/W)\xrightarrow{\Phi_l^*}\ ^{\Phi_l^*\vec{\pi}^*\tau}\widetilde{K}^{*}\big(((\mathfrak{t}\times_{\pi_{1}T}\Lambda)\times\mathbb{Z}_{\leq0})_{+}/W\big)\oplus  \bigoplus_{n<0}\ ^{\Phi_l^*\vec{\pi}^*\tau}\widetilde{K}^{*}\big((\mathfrak{t}\times_{\pi_{1}T}\Lambda)\times\{n\})_{+}/W\big)$$
$$\simeq\ ^{\Phi_l^*\vec{\pi}^*\tau}\widetilde{K}^{*}\big((\mathfrak{t}\times_{\pi_{1}T}\Lambda)_+/W\big)[[q]]\oplus\bigoplus_{n<0}\  ^{\Phi_l^*\vec{\pi}^*\tau}\widetilde{K}^{*}\big((\mathfrak{t}\times_{\pi_{1}T}\Lambda)_+/W\big)q^{-n}$$
$$\simeq\ ^{\Phi_l^*\vec{\pi}^*\tau}\widetilde{K}^{*}\big((\mathfrak{t}\times_{\pi_{1}T}\Lambda)_+/W\big)((q)).$$
Given that the $\mathbb{Z}((q))$-module structure of $^{\vec{\pi}^*\tau}K^{*}(\vec{P}_{\tau}/W,T_\infty/W)$ comes via pullback along the shift map sh, it is clear that the $\mathbb{Z}((q))$-module structure suggested by the last displayed abelian group is actually the correct one.

It now remains to define $\tau'$ and identify $^{\Phi_l^*\vec{\pi}^*\tau}\widetilde{K}^{*}\big((\mathfrak{t}\times_{\pi_{1}T}\Lambda)_+/W\big)((q))$ and $^{\tau'}\vec{K}_{S^1}^{*+\text{dim}T}((\Lambda/\pi_{1}T)/W)$ as $\mathbb{Z}((q))$-modules. Recall that a vector bundle induces a twist of the $K$-theory of the base of that vector bundle (c.f. [FHT1] 3.6). Note that $(\mathfrak{t}\times_{\pi_{1}T}\Lambda)_+$ is the Thom space of the vector bundle $\mathfrak{t}\times_{\pi_{1}T}\Lambda\rightarrow \Lambda/\pi_{1}T$, and call the corresponding twist $\sigma(\mathfrak{t})$. Then applying the Thom isomorphism (c.f. [FHT1] 3.6) in $W$-equivariant twisted $K$-theory gives
$$^{\Phi_l^*\vec{\pi}^*\tau}\widetilde{K}^{*}\big((\mathfrak{t}\times_{\pi_{1}T}\Lambda)_+\big)((q))\simeq\ ^{\Phi_l^*\vec{\pi}^*\tau-\sigma(\mathfrak{t})}\widetilde{K}^{*-\text{dim}\mathfrak{t}}\big(\Lambda/\pi_1T\big)((q)).$$
Finally, write $\tau'=\Phi_l^*\vec{\pi}^*\tau-\sigma(\mathfrak{t})$. Since $(\Lambda/\pi_{1}T)/W$ is a trivial $B\mathbb{Z}$-groupoid (c.f. Definition \ref{BZdef}) 
$$^{\tau'}\vec{K}_{S^1}^{*}((\Lambda/\pi_{1}T)/W) \simeq\ ^{\tau'}K^{*}((\Lambda/\pi_{1}T)/W)\otimes_{\mathbb{Z}}\mathbb{Z}((q)).$$
\end{proof}

\begin{cor} \label{singlecongsupp} At negative level $^{\tau} \vec{K}^{*}_{S^1}(T/N)$ is spanned by classes supported at single conjugacy classes.
\end{cor}
\begin{proof}
This is immediate from the proof of Lemma \ref{abelianizednegative}, since the classes in $^{\tau_w} \vec{K}^{*}_{S^1}(T/N)$ are all images of degree zero classes under a Thom isomorphism from a vector bundle whose dimension is equal to the dimension of $T$. 
\end{proof}

\begin{lem}\label{negfinalcalc} Let $G$ be a compact Lie group with identity component $G(1)$. Let $\Lambda_\text{reg}$ be the regular weights of $G$, i.e. those on which are not fixed by a Weyl reflection $r\in W=N/T$. For negative twists $\tau\in H^3_{S^1}(G/G)$, there is a twist $\tau'\in H^3_{S^1}(\Lambda_{\text{reg}}/\pi_{1}T/W)$ (described in the proof) such that $\omega^*$ induces an isomorphism $$\omega^*:\ ^{\tau}\vec{K}^{*}_{S^1}(G(1)/G)\xrightarrow{\sim}\ ^{\tau'}\vec{K}_{S^1}^{*}(\Lambda_{\text{reg}}/\pi_{1}T/W)\subset\ ^{\tau'}\vec{K}_{S^1}^{*}((\Lambda/\pi_{1}T)/W).$$ 
\end{lem}
\begin{proof} 
As usual, consider a point $t\in T$ with stabilizer $G_t\subset G$. Write $T_t(=T)$ for a maximal torus of $G_t$ and let $N_t\subset N$ be the normalizer of $T_t$ in $G_t$. Write $\tau(t)$ for the restriction of $\tau$ to $\{t\}/G_t$. Note that $\omega:T/N\rightarrow G(1)/G$ is essentially surjective. By the definition of $\omega$ (c.f. Lemma \ref{omegapush}) it suffices to work at a fixed $q$-power. By the same Lemma, the local behavior of $\omega$, denoted $(\omega_t)_*$, is the $G_t$-equivariant Becker-Gottlieb transfer along $G_t/N_t\rightarrow\text{pt}$. By [N] Lemma 4.3 the composition $(\omega_t)_{*}\omega_t^{*}$ is multiplication by the $G_t$-equivariant Euler characteristic $\chi(G_t/N_{t})=|N_t|/|N_{t}|=1$. So the map $\omega_{*}\omega^{*}$ is locally the identity. By the Mayer-Vietoris axiom (and equivariant local contractibility of all groupoids involved) it follows that it is a global isomorphism. Then, in light of Corollary \ref{singlecongsupp}, the image of any class under $\omega_{*}$ is fixed by $\omega_{*}\omega^{*}$, so $\omega_{*}\omega^{*}$ is a projection, and so must be the identity. Therefore $^{\tau}\vec{K}^{*}_{S^1}(G/G)$ is split inside $^{\omega^*\tau} \vec{K}^{*}_{S^1}(T/N)$. 

It remains to identify the summand. For that it suffices to identify the kernel of $\omega_*$ with the classes supported away from regular conjugacy classes. Recall from Corollary \ref{singlecongsupp} that $^{\omega^*\tau} \vec{K}^{*}_{S^1}(T/N)$ is spanned by classes supported at single conjugacy classes. Moreover, for any $t\in T$ all classes supported at $\{t\}/N_t\subset T/N$ are of the form $[V\otimes \Theta_t]$ where $V\in R^{\omega^*\tau(t)}(N_{t})$ and $\Theta_t$ represents the euler class of the normal bundle of $\{t\}/N_{t}\hookrightarrow T/N$, which depends on a choice of complex orientation but is always a class of virtual dimension zero in $R(N_t)$. Finally, under the identification of $^{\omega^*\tau} \vec{K}^{*}_{S^1}(T/N)$ with $^{\tau'}\vec{K}_{S^1}^{*}(\Lambda/\pi_{1}T/W$, the subgroup $^{\tau'}\vec{K}_{S^1}^{*}(\Lambda_{\text{reg}}/\pi_{1}T/W)$ corresponds to classes supported at $t\in T$ that define regular conjugacy class in $G$ (i.e. $G_t=N_t$). Recall that $(\omega_t)_*$ is given by Segal induction, power-by-power in $q$ (c.f. the proof of Lemma \ref{omegapush}). Clearly if $G_t=N_t$ then $(\omega_t)_{*}$ is the identity. If $t$ does not define a regular conjugacy class in $G$, Then $W_t=N_t/T_t$ contains a Weyl reflection, and the proof is reduced to Lemma \ref{segalkernel}.
\end{proof}

\subsection{Calculation of $^{\tau}\vec{K}^{*}_{S^1}(G/G)$ at Positive Twist}

\begin{lem}\label{abelianizedpositive}
Let $\tau\in H^3_{S^1}(T/N)$ be a positive twist. Equip $(\Lambda/\pi_{1}T)/W$ with the trivial $B\mathbb{Z}$-action (c.f. Definition \ref{BZdef}). There is a twist $\tau'\in H^3_{S^1}((\Lambda/\pi_{1}T)/W)$ (described in the proof) and an isomorphism of $\mathbb{Z}((q))$-modules
$$^{\tau} \vec{K}_{S^1}^{*}(T/N)\xrightarrow{\sim}\ ^{\tau'}\vec{K}_{S^1}^{*}((\Lambda/\pi_{1}T)/W).$$

\end{lem}
\begin{proof} The proof is similar to that of Lemma \ref{abelianizednegative}. The plan is to invoke Lemma \ref{keylemma}, recognize the homotopy type of $\vec{P}_\tau$ as the disjointly base-pointed total space of a vector bundle over $(\Lambda/\pi_{1}T)/W$, and then equivariantly contract the fibers of that vector bundle and discard the disjoint base-point.

I claim that because $\tau$ is positive, the complement of any closed, connected, convex set $C\subset P_\tau$ contributes to the topology of $\vec{P}_{\tau}$ (c.f. Definition \ref{compactification}). It suffices to consider the case that $C$ is the image of $\mathfrak{t}\times\{\lambda\}\times\{n\}$ for some $(\lambda,n)$. This certainly satisfies Condition 1 listed in Definition \ref{compactification}. For any $t\in T$, and lift $X$ of $t$, the inverse of the induced map $\pi^{-1}(t)\leftarrow\Lambda^{\tau(t)}\times\mathbb{Z}$ sends $[(X-p,\lambda,n)]$ to $(\lambda-\kappa^\tau p,n+\tau(\kappa^\tau p,\kappa^\tau p)+\lambda(p))$, whose second coordinate is a positive quadratic function of $p$ and so certainly satisfies Condition 2. 

It follows that for any bounded below $S\subset \mathbb{Z}$, the image of $\mathfrak{t}\times\{\lambda\}\times S$ in $P_\tau$ satisfies Conditions 1 and 2. This leads to the following.

Let $\vec{\mathbb{Z}}$ denote the following partial compactification of $\mathbb{Z}$: as a set $\vec{\mathbb{Z}}=\mathbb{Z}\coprod\{\infty_+\}$ and the open neighborhoods of $\{\infty_+\}$ are defined to be the complements of sets $C\in\mathbb{Z}$ which are bounded below. As in the proof of Lemma \ref{abelianizednegative}, define an action of $\pi_1T$ on $\mathfrak{t}\times\Lambda$ by $p(X,\lambda)=(X+p,\lambda-\kappa^\tau p)$. Let $\big(\mathfrak{t}\times_{\pi_{1}T}\Lambda\coprod \{*\}\big)$ be the evident addition of a disjoint basepoint with trivial $W$-action. For any $(X,\lambda)\in\mathfrak{t}\times\Lambda$ let $(X_l,\lambda_{l}$ be the $\pi_{1}T$-translate such that $\lambda_l$ that lies the fundamental domain containing $l$. Then the pointed map 
$$\Phi_{l}: \big(\mathfrak{t}\times_{\pi_{1}T}\Lambda^{\tau}\coprod \{*\}\big)\wedge\vec{\mathbb{Z}}\rightarrow\vec{P}_{\tau}/T_{\infty}$$ 
$$([X,\lambda], n)\mapsto [\nu_{l}(X), \lambda_{l}, n]$$
is a $W$-equivariant homeomorphism. Hence $^{\vec{\pi}^*\tau}\widetilde{K}^{*}(\vec{P}_{\tau}/W)\simeq\ ^{\Phi_l^*\vec{\pi}^*\tau}K^{*}\big((\mathfrak{t}\times_{\pi_{1}T}\Lambda^{\tau})/W\big)\otimes_{\mathbb{Z}}\mathbb{Z}((q))$. The proof of the isomorphism displayed in the lemma is completed by defining $\tau'=\Phi_l^*\vec{\pi}^*\tau$, noting that $\mathfrak{t}\times_{\pi_{1}T}\Lambda^{\tau}$ equivariantly deformation retracts onto $\Lambda/\pi_{1}T$, and that (as in the proof of Lemma \ref{abelianizednegative})
$$^{\tau'}\vec{K}_{S^1}^{*}((\Lambda/\pi_{1}T)/W) \simeq\ ^{\tau'}K^{*}((\Lambda/\pi_{1}T)/W)\otimes_{\mathbb{Z}}\mathbb{Z}((q)).$$
\end{proof}

\begin{defn}\label{PERdef} ([PS] Chapter 9)
For a compact Lie group $G$ and an element $\tau\in H^3_{S^1}(LBG)$ write $LG^\tau\rtimes S^1\rightarrow LG\rtimes S^1$ for the associated central $U(1)$-extension, and write $\hat{R}_{\text{pos}}^{\tau}(LG\rtimes S^1)$ for the free $\mathbb{Z}((q))$-module consisting of positive energy, $\tau$-projective (also know as `level $\tau$') representations of $LG\rtimes S^1$.
\end{defn}

\begin{lem}\label{PERclasses} For positive twists $\tau$ there is an injection of $\mathbb{Z}((q))$-modules $\hat{R}_{\text{pos}}^{\tau}(LG\rtimes S^1)\hookrightarrow\ ^{\tau} \vec{K}_{S^1}^{0}(G(1)/G)$.
\end{lem}
\begin{proof} Let $\mathcal{H}\in\hat{R}_{\text{pos}}^{\tau}(LG\rtimes S^1)$ be an irreducible level $\tau$ positive energy representation. Let $S\subset  \Lambda\times\mathbb{Z}$ be the weights of $T\times S^1\subset LG\rtimes S^1$ appearing in $\mathcal{H}$. By [PS] Theorem 9.3.5 this has a unique lowest weight $(\lambda, n)\in S$. The affine Weyl group $W\ltimes\pi_1T$ acts on $S$, and the subgroup $\pi_1T$ acts by the first displayed formula in the proof of Lemma \ref{ptau} (c.f. formula (9.3.3) in [PS]). Hence, since $\tau$ is positive, for any $t\in G$ with stabilizer $G_t$ the infinite dimensional (projective) representation $\mathcal{H}$ produces a well-defined element of $^{\tau(t)}\vec{K}_{S^1}^*(\{t\}/G_t)$ (c.f. Lemma \ref{stalksofKhat}). Then since $(G(1)/G)/B\mathbb{Z}\simeq\mathcal{A}(G)/(LG\rtimes S^1)$ ([FHT2], Section 2.1), the construction
$$\mathcal{A}(G)\times_{(LG\rtimes S^1)}\mathcal{H}_{\lambda}$$
defines the required map. It is injective since the pullback along the inclusion $1\hookrightarrow G$ detects the character of $\mathcal{H}_{\lambda}$ as an element of $R^{\tau(1)}(G)((q))\simeq\ ^{\tau(1)} \vec{K}_{S^1}^{0}(\text{pt}/G)$ (c.f. Lemma \ref{stalksofKhat}).
\end{proof}

\begin{defn}\label{qdim0}
Let $\mathbf{X}$  be a $B\mathbb{Z}$-groupoid let $\tau\in H^3_{S^1}(\mathbf{X})$ be a twist. A class $\xi\in\ ^\tau\vec{K}^*_{S^1}\mathbf{X}$ is said to have \textit{virtual $q$-dimension 0} if for any point $x\in X_0$ with stabilizer $G_x$, pullback along $i_x:\{x\}/G_x\rightarrow\mathbf{X}$ produces a class $i_x^*\xi\in\ ^{i_x^*\tau}\vec{K}_{S^1}^*(\{x\}/G_x)$ which maps to zero under the restriction map $^{i_x^*\tau}\vec{K}_{S^1}^*(\{x\}/G_x)\rightarrow \vec{K}_{S^1}^*(\{x\})=\mathbb{Z}((q))$ induced by $1\rightarrow G_x$ (c.f. Lemma \ref{stalksofKhat}).
\end{defn}

\begin{lem}\label{qdim0ker} Let $\tau$ be a positive twist. Identify $^{\tau} \vec{K}^{*}_{S^1}(T/N)$ with $^{\tau'}\vec{K}_{S^1}^{*}((\Lambda/\pi_{1}T)/W)$ as in the previous lemma. Then the kernel of the map 
$$\omega_{*}:\ ^{\tau'}\vec{K}_{S^1}^{*}((\Lambda/\pi_{1}T)/W)\longrightarrow\ ^{\tau} \vec{K}_{S^1}^{*}(G/G)$$ consists of the classes of virtual $q$-dimension 0 (c.f. Definition \ref{qdim0}).
\end{lem}
\begin{proof} For $t\in N$ write $N_t$ for its stabilizer in $N$, $G_t$ for its stabilizer in $G$, and $T_t$ for a maximal torus of $G_t$. Recall from the proof of Lemma \ref{omegapush} that $\omega_*$ is defined power-by-power in $q$, and the local model at a point $t\in G$ with stabilizer $G_t\in G$ and  coincides with Segal induction $\iota_{!}:R(N_{t})\rightarrow R(G_t)$. 

Now let $[V]\in\ ^{\tau}\vec{K}_{S^1}^{*}((\Lambda/\pi_{1}T)/W)$ be a class of virtual $q$-dimension 0. Then its support must be contained in the set of $[\lambda]\in\Lambda/\pi_{1}T$ with nontrivial $W$-stabilizer. Over each such point its fiber is a representation of the stabilizer of virtual $q$-dimension 0. It suffices to assume that $[V]$ is supported at a single such point $[\lambda]$ with stabilizer $W_\lambda$. Tracing $[V]$ backwards through the construction of Lemma \ref{abelianizedpositive}, there is a point $t\in T$ with stabilizer $N_{t}$ such that the class in $^{\tau} \vec{K}^{*}_{S^1}(T/N)$ corresponding to $[V]$ is represented by a Hilbert bundle $\mathcal{V}$ whose fiber at $t$ can be identified (by a choice of fractional splitting $\psi:\tilde{N}_{t}\leftarrow S^1_d$) with a Laurent series in $q$ with coefficients in ($\tau$-projective) virtual representations of virtual dimension 0. Since $(\omega_t)_*$ is defined power-by-power, Lemma \ref{segalkernel} implies that $(\omega_t)_*\mathcal{V}_t$ has $q$-dimension zero at $t$. Hence $\omega_*[V]$ and $\omega^*\omega_*[V]$ have $q$-dimension zero.
But from the free $\mathbb{Z}((q))$-basis of $^{\tau} \vec{K}^{*}_{S^1}(T/N)$ established in Lemma \ref{abelianizedpositive}, it is evident that the only class which is of $q$-dimension 0 at any point is the zero class, So $\omega^*\omega_*[\mathcal{V}]=0$. Injectivity of $\omega^{*}$ implies that $\omega_{*}[\mathcal{V}]=0$. It remains to prove that the classes of virtual dimension zero span the kernel of $\omega_{*}$. By Lemma \ref{PERclasses}, the image of $\omega^{*}$ contains a subspace isomorphic to $\hat{R}_{\text{pos}}^{\tau}(LG\rtimes S^1)$. Since $\omega_{*}\omega^{*}$ is an isomorphism (c.f. the proof of Lemma \ref{negfinalcalc}), the proof is completed by a counting argument using the following result.
\end{proof}

\begin{thm}(c.f. [PS] Theorem 9.3.5, [FHT3] Theorem 10.2) A free $\mathbb{Z}((q))$-basis of $\hat{R}_{\text{pos}}^{\tau}(LG\rtimes S^1)$ consisting of irreducibles is in one-to-one correspondence with the set of orbits of the action of the affine Weyl group $W\ltimes\pi_1T$ on $\Lambda$ defined by $\tau$ (c.f. Lemma \ref{ptau}), i.e. the coarse quotient $[(\Lambda^{}/\pi_{1}T)/W]$.
\end{thm}
\noindent
The preceding two results immediately imply the following.
\begin{cor}\label{posfinalcalc} For positive $\tau$, $^{\tau} \vec{K}_{S^1}^{*}(G/G)\simeq\hat{R}_{\text{pos}}^{\tau}(LG\rtimes S^1)$, where the right hand side is viewed as a graded $\mathbb{Z}((q))$-module concentrated in degree zero.
\end{cor}

\subsection{Reconciliation of Bases}

\begin{defn}\label{spinors}
Let $G$ be a compact Lie group. Let $\text{Cliff}(\mathfrak{g}^*)$ denote the Clifford algebra of the vector space $\mathfrak{g}^*$ equipped with the bilinear Killing form. By integrating over $S^1$, the Killing form induces a bilinear form on $L\mathfrak{g}^*=C^\infty(S^1,\mathfrak{g}^*)$. Let $\text{Cliff}(L\mathfrak{g}^*)$ denote the associated Clifford algebra, and let $\mathcal{S}^\pm(L\mathfrak{g}^*)$ be an irreducible $\mathbb{Z}/2$-graded representation of $\text{Cliff}(L\mathfrak{g}^*)$. If $\mathcal{S}^\pm(0)$ is an irreducible $\mathbb{Z}/2$-graded representation of $\text{Cliff}(\mathfrak{g}^*)$ then $\mathcal{S}^\pm(L\mathfrak{g}^*)$ may be presented as $\mathcal{S}^\pm(0)\otimes \bigwedge^\text{ev/odd}(z\mathfrak{g}_\mathbb{C}[z])$ (c.f. [FHT3] 8.6). It admits a projective, positive energy action of $LG$, and the corresponding level is denoted by $\sigma$ (c.f. [FHT3] 1.6, 8.8).
\end{defn}

While it is not curious that the calculation at positive and negative twist differ\footnote{See the discussion at the beginning of Section 3.} it is curious \textit{how} they differ. At positive level the representation ring $\hat{R}_{\text{pos}}^{\tau}(LG\rtimes S^1)$ makes an appearance as $^{\tau} \vec{K}_{S^1}^{0}(G/G)\simeq\ ^{\tau} \vec{K}_{S^1}^{\text{dim}G}(\Lambda^{\tau}/W)$, while at negative level, the isomorphic ring $\hat{R}_{\text{neg}}^{-\tau}(LG\rtimes S^1)$ makes its appearance as $^{\tau+\sigma} \vec{K}_{S^1}^{\text{dim}G}(G/G)\simeq\ ^{\tau+\sigma} \vec{K}_{S^1}^{\text{dim}G}(\Lambda^{\tau+\sigma}_{\text{reg}}/W)$. Thus at positive level there is no shift in twist and all $W$-orbits contribute a basis element, while at negative level there is a shift by $\sigma$ in the twist and only the \textit{regular} orbits contribute basis elements. The correspondence becomes even more curious when related to the correspondence for the maximal torus, where there is no $\sigma$-discrepancy between positive and negative twists. 

This is made precise in the following lemma.

\begin{lem}\label{duality} Let $G$ be a simple and simply-connected Lie group. Fix a positive twist $\tau\in H^3_{S^1}(G/G)$. There are dualities of finitely-generated free (ungraded) $\mathbb{Z}((q))$-modules
$$\mathbb{D}_T: \ ^{-\tau-\sigma}\vec{K}_{S^1}^{r}(T/N)\otimes\ ^{\tau+\sigma}\vec{K}_{S^1}^{0}(T/N)\longrightarrow \mathbb{Z}((q)),$$
$$\mathbb{D}_G: \ ^{-\tau-\sigma} \vec{K}_{S^1}^{r}(G/G)\otimes\ ^{\tau}\vec{K}_{S^1}^{0}(G/G)\longrightarrow\mathbb{Z}((q)).$$
\end{lem}
\begin{proof} 
By Corollary \ref{posfinalcalc}, the map $\hat{R}_{\text{pos}}^{\tau}(LG\rtimes S^1)\hookrightarrow\ ^{\tau} \vec{K}_{S^1}^{0}(G/G)$ in Lemma \ref{PERclasses} is an isomorphism. The content of Section 12 and 13 of [FHT3] (c.f. 12.9, Proposition 13.6) is that for each $[\mathcal{H}_\lambda]\in\hat{R}_{\text{pos}}^{\tau}(LG\rtimes S^1)$ there is a $G$-parametrized family $F_\lambda$ of Fredholm operators on $\mathcal{H}_\lambda\otimes\mathcal{S}^\pm$ such that the assignment $[\mathcal{H}_\lambda]\mapsto (\mathcal{H}_\lambda\otimes\mathcal{S}^\pm,F_\lambda)$ defines an isomorphism $\hat{R}_{\text{pos}}^{\tau}(LG\rtimes S^1)\hookrightarrow\ ^{-\tau-\sigma} K_{S^1}^{r}(G/G)$. By Lemma \ref{negfinalcalc} the natural map $\ ^{-\tau-\sigma} K_{S^1}^{r}(G/G)\rightarrow\ ^{-\tau-\sigma}\vec{K}_{S^1}^{r}(G/G)$ is an isomorphism.
Therefore the span of isomorphisms of finite rank, free $\mathbb{Z}((q))$-modules
$$^{-\tau-\sigma}\vec{K}_{S^1}^{r}(G/G)\leftarrow\hat{R}_{\text{pos}}^{\tau}(LG\rtimes S^1)\rightarrow\ ^{\tau} \vec{K}_{S^1}^{0}(G/G)$$
exhibits a duality $\mathbb{D}_G$ between the left and right terms, since the middle is canonically self-dual (it has a canonical basis of irreducibles).

Let $\hat{R}_{\text{pos}}^{\tau}(LT\rtimes S^1)^{hW}$ be the $\mathbb{Z}((q))$-module of isomorphism classes of  $W$-\textit{equivariant}\footnote{As opposed to \textit{invariant}.} $\tau$-twisted positive energy $LT$ representations. The same recipe as in the previous paragraph produces a span of isomorphisms of finite rank, free $\mathbb{Z}((q))$-modules
$$^{-\tau-\sigma}\vec{K}_{S^1}^{r}(T/N)\leftarrow\hat{R}_{\text{pos}}^{\tau}(LT\rtimes S^1)^{hW}\rightarrow\ ^{\tau+\sigma} \vec{K}_{S^1}^{0}(T/N)$$
exhibits a duality $\mathbb{D}_T$ between the left and right terms, since the middle is again canonically self-dual. 
\end{proof}

The reconciliation comes in the form of the following lemma, which is preceded by the following definition.

\begin{lem} Let $G$ be a simple and simply-connected Lie group. Fix a positive twist $\tau\in H^3_{S^1}(G/G)$. The dualities of the previous lemma are intertwined by maps
$$\omega_{*}:\ ^{-\tau-\sigma}\vec{K}_{S^1}^{r}(T/N)\rightarrow\ ^{-\tau-\sigma} \vec{K}_{S^1}^{r}(G/G),$$
$$\omega^{!}:\ ^{\tau}\vec{K}_{S^1}^{0}(G/G)\rightarrow\ ^{\tau+\sigma}\vec{K}_{S^1}^{0}(T/N).$$
The first map is the pullback along $\omega: T/N\hookrightarrow G/G$ followed by tensoring with the spinor representation $\mathcal{S}^{\pm}(L\mathfrak{g}^*)$ (c.f. Definition \ref{spinors}). The second is the Becker-Gottlieb transfer along $\omega$.
\end{lem}
\begin{proof} The claim is that for any $a\in\ ^{-\tau-\sigma}\vec{K}_{S^1}^{r}(T/N)$ and $b\in\ ^{\tau}\vec{K}_{S^1}^{0}(G/G)$, $\mathbb{D}_G(\omega_*a\otimes b)=\mathbb{D}_T(a\otimes\omega^!b)$. Since $\omega_*\omega^*$ is the identity at negative level (c.f. the proof of Lemma \ref{negfinalcalc}), it suffices to show that for any $c\in\ ^{-\tau-\sigma}\vec{K}_{S^1}^{r}(G/G)$ and $b\in\ ^{\tau}\vec{K}_{S^1}^{0}(G/G)$, $\mathbb{D}_G(c\otimes b)=\mathbb{D}_T(\omega^*c\otimes\omega^!b)$. 

Under the identifications of the previous lemma the restriction maps 
$\omega^*:\ ^{-\tau-\sigma}\vec{K}_{S^1}^{r}(G/G)\rightarrow\ ^{-\tau-\sigma}\vec{K}_{S^1}^{r}(T/N)$ and
$\omega^*:\ ^{\tau}\vec{K}_{S^1}^{0}(G/G)\rightarrow\ ^{\tau}\vec{K}_{S^1}^{0}(T/N)$
both agree with the restriction of representations $\hat{R}_{\text{pos}}^{\tau}(LG\rtimes S^1)\rightarrow\hat{R}_{\text{pos}}^{\tau}(LT\rtimes S^1)^{hW}$. The equation $\mathbb{D}_G(c\otimes b)=\mathbb{D}_T(\omega^*c\otimes\omega^!b)$ now follows since the discrepancy between $\omega^!$ and $\omega^*$ is precisely the tensor product with $\mathcal{S}^\pm$ that is missing on the positive level side.
\end{proof}

\section{Equivariant Elliptic Cohomology}

In this section the completed $\vec{K}_{S^1}$-theory defined above is used to give a $K$-theoretic picture of equivariant elliptic cohomology at the Tate curve. Despite being the shortest section, it is in some sense the paper's centerpiece. In fact, its shortness speaks to the wonderful simplicity of the Kitchloo-Morava picture of elliptic cohomology at the Tate curve. Fix a compact Lie group $G$. A $G$-\text{equivariant elliptic cohomology theory} is defined to be:
 
\begin{enumerate}[itemsep=0pt]
\item a weakly even $G$-equivariant cohomology theory $E_{G}$, 
\item an elliptic curve $\mathcal{E}$ over $E^{0}(\text{pt})=E^{0}_{G}(G)$,

\item a twist $\tau\in H^{4}(BG)$ for $E_{G}$ with associated transgressed class $\text{tr}(\tau)\in H^3_{S^1}(LBG)$ and central extension $LG^\tau\rtimes S^1\rightarrow LG\rtimes S^1$
\item an isomorphism of $\mathbb{Z}((q))$-modules $^{\tau} E_{G}^{0}(\text{pt})\rightarrow \hat{R}^\tau_\text{pos}(LG\rtimes S^1)$ (c.f. Definition \ref{PERdef})
\item and an isomorphism of formal groups Spf$E^{0}(\mathbb{CP}^{\infty})= E^{0}_{G}(G\times\mathbb{CP}^{\infty})\rightarrow\hat{\mathcal{E}}$,

\end{enumerate}

\subsection{The Equivariant Kitchloo-Morava Construction}

Let $G\text{-Spaces}_\text{rel}$ be the category of pairs $(M,A)$ where $M$ is a $G$-space and $A\subset M$ is $G$-invariant. A class $\tau\in H^4BG$ transgresses to a class $\text{tr}(\tau)\in H^3_{S^1}LBG(=H^3_{S^1}(\mathcal{L}(\text{pt}/G))$ that restricts to zero in $H^3\Omega BG=H^3G$ and hence defines a graded central extension (with trivial grading) $\mathbf{L}^\tau\rightarrow \mathcal{L}(\text{pt}/G)/B\mathbb{Z}$ (c.f. Remark \ref{deg3twist}). Write $p_M$ for the projection $M\rightarrow\text{pt}$. Recall that $\mathcal{L}(M/G)$ has a natural $B\mathbb{Z}$-action (c.f. Example \ref{loopgroupoid}). Define a functor $J_\tau: G\text{-Spaces}_\text{rel}\rightarrow B\mathbb{Z}\text{-}\mathfrak{Twist}_\text{rel}$ by the formula
$$J_\tau(M,A):= (\mathcal{L}(M/G)/B\mathbb{Z},\mathcal{L}(A/G)/B\mathbb{Z},\mathcal{L}(p_M)^*\mathbf{L}^\tau).$$

\begin{prop} For any compact connected Lie group $G$ and $\tau\in H^4BG$ whose transgression to $H^3_{S^1}(LBG)$ defines a strongly topologically regular positive twist the composite functor 
$$^{\tau}\text{E}^{*}_{G}:=\vec{K}_{S^1}^*\circ J_\tau:G\text{-Spaces}_\text{rel}\longrightarrow\mathbb{Z}((q))\text{-mod}$$
$$(M, A)\mapsto\ ^{\tau}\text{E}^{*}_{G}(M,A)$$
defines a $G$-equivariant elliptic cohomology theory at the Tate curve.
\end{prop}
\begin{proof} Given that $\mathcal{L}((M\setminus A)/G)\simeq\mathcal{L}(M/G)\setminus\mathcal{L}(A/G)$, the cohomology axioms follow immediately from their holding for $\vec{K}_{S^1}$. It remains to establish ellipticity as defined above. Take $\mathcal{E}$ to be the Tate curve over $\mathbb{Z}$. By hypothesis $\tau$ defines a strongly topologically regular \textit{positive} twist, so Corollary 3.2.5 says precisely that  $\ ^{\tau} E_{G}^{0}(\text{pt})\simeq \hat{R}_{\text{pos}}^{\tau}(LG_1\rtimes S^1)$. 
\end{proof}

\begin{rem}
When $G$ is disconnected all the definitions still make sense; instead the restriction to connected groups is to preserve the maximum amount of correlation between $^\tau\vec{K}_{S^1}(G/G)$ and the positive energy representation theory of $LG$. Indeed, when $G$ is disconnected one has to introduce \textit{twisted loop groups}\footnote{The `twisted' here is unrelated to the `twisted' in `twisted $K$-theory.'} (c.f. [FHT3] 1.5) and their representation theory, which results in an explosion of notational complexity and bookkeeping that obscures a major virtue of the Kitchloo-Morava construction, namely its simplicity.
\end{rem}

\begin{rem}
Some readers might prefer if item 4 were replaced with something like the following condition: a twist $\tau\in H^{4}(BG)$ for $E_{G}$ with its associated line bundle $\mathcal{L}_{\tau}\rightarrow\text{Bun}_{G}(\mathcal{E})$ over the moduli space of principal $G$-bundles over $\mathcal{E}$. If $G$ is simple and simply-connected then by [A] Theorem D, the Kac character map gives an isomorphism of $\hat{R}_{\text{pos}}^{\tau}(LG_1\rtimes S^1)$ with $\Gamma(\text{Bun}_{G}(\mathcal{E}); \mathcal{L}_{\tau})$. I have chosen to require only a relation to positive energy representation theory so as not to have to invoke the Kac character theory to prove ellipticity and to treat all connected Lie groups uniformly. 
\end{rem}

\begin{rem}
At negative level the construction is still plagued by the problems mentioned in the introduction of this paper and so does not produce an elliptic cohomology theory. That is no surprise, since the completion that defines $\vec{K}_{S^1}$ favors \textit{positive} $q$-powers. On the other hand, the functor $^\tau E_G^*$ makes perfect sense for negative $\tau$ and still produces a $G$-equivariant cohomology theory, just not an elliptic one (at least as per the 5 specifications listed above). This motivates the following section.
\end{rem}

\subsection{Duality in $E_{G}$}
When $G$ is simple and simply-connected, Lemma \ref{duality} immediately implies the following.

\begin{lem} Let $G$ be a simple and simply-connected Lie group of dimension $d$ and let $\sigma$ be the twist associated to the positive energy spin representation of the adjoint representation of $LG$ (c.f. Definition \ref{adjointspin}). There is a natural duality pairing
$$^{-\tau-\sigma}E_{G}^{d}(\text{pt})\otimes\ ^{\tau}E_{G}^{0}(\text{pt})\longrightarrow\  ^{-\sigma}E_{G}^{d}(\text{pt})\simeq\mathbb{Z}((q)).$$
\end{lem}

\subsection{Comparisons}

When $G$ is connected, Grojnowski [G] has constructed the seminal ``delocalised equivariant elliptic cohomology" over the complex numbers as follows (c.f. also [BET] 4.1). The theory takes values in holomorphic sheaves over $\mathbb{H}\times\mathfrak{t}_\mathbb{C}/(W\ltimes\pi_1T^2)$. Recall that every $\tau\in\mathbb{H}$ defines an isomorphism $\mathbb{C}\simeq \mathbb{R}\times \mathbb{R}$. Write $\phi$ for the composite (c.f. [BET] Display (16))
$$\mathbb{H}\times\mathfrak{t}_\mathbb{C}\xrightarrow{\text{canonical}}\mathbb{H}\times \mathfrak{t}\times\mathfrak{t}\xrightarrow{id\times\text{exp}}\mathbb{H}\times T\times T\xrightarrow{\text{pr}} T\times T.$$ 
For a $G$-space $X$ and a point $a\in \mathbb{H}\times\mathfrak{t}_\mathbb{C}$ write $X^a$ for the common fixed point set of the two components of $\phi(a)$, $Z(a)\subset G$ for their common centralizer, and $W_a$ for the Weyl group of $Z(a)$ (which is a subgroup of the Weyl group of $G$). Write $H_T^*(-;\mathbb{C})$ for equivariant singular cohomology with complex coefficients and mod 2 grading and $\mathcal{O}$ for the sheaf of holomorphic functions on both $\mathfrak{t}_\mathbb{C}$ and $\mathbb{H}\times\mathfrak{t}_\mathbb{C}$. Let $U$ be a small neighborhood of $0\in \mathfrak{t}_\mathbb{C}$. Since $H_{T}^*(\text{pt};\mathbb{C})$ is canonically isomorphic to Sym$(\mathfrak{t}_{\mathbb{C}}^{\vee}[-2])$ it acts on $\mathcal{O}(U)$. Since $H_{Z(a)}^*(\text{pt};\mathbb{C})\simeq H_{T}^*(\text{pt};\mathbb{C})^{W_a}$ the latter acts on $\mathcal{O}(U_{a})^{W_a}$.
Note that the second component of $a$ acts on the second factor of $\mathbb{H}\times\mathfrak{t}_\mathbb{C}$ by translation. The action of $W$ and $\pi_1T$ on $\mathfrak{t}$ combine to give a $W\ltimes\pi_1T^2$ action on $\mathbb{H}\times\mathfrak{t}\times\mathfrak{t}$ that is trivial on the first factor. Pulling back along the map `canonical' gives an action of $W\ltimes\pi_1T^2$ on $\mathbb{H}\times\mathfrak{t}_\mathbb{C}$.  Let $U_{a}$ be a sufficiently small $W\ltimes\pi_1T^2$-invariant neighborhood of $a$. For $\tau\in H^3_{S^1}(G/G)$, the associated quadratic form on $H_1(T)$ (c.f. Definition \ref{strongtopreg}) defines an element of $H^2(T\times T)$, and let $\mathcal{L}^\tau$ denote the associated `Looijenga' line bundle. Define the sheaf $^\tau(\textit{Ell}_G^\text{Groj})^*X$ at $U_a$ to be
$$(\textit{Ell}_G^\text{Groj})^*X(U_{a}):=(a^{-1})^*\Big(H_{Z(a)}^*(X^{a};\mathbb{C})\otimes_{H_{Z(a)}^*(\text{pt};\mathbb{C})}a^{*}(\mathcal{O}(U_{a})\otimes\mathcal{L}(U_a))^{W_a}\Big)$$
The evident restriction maps for $W\ltimes\pi_1T^2$-invariant subsets $U_{a'}\subset U_a$ centered around $a'\in U_a$ are isomorphisms by the localization theorem in equivariant cohomology and the fact that fixed point sets  can only shrink locally (i.e. if $g$ does not stabilize $x$ then it does not stabilize a neighborhood of $x$).

\begin{lem}\label{grojcomp} Let $G$ be a connected Lie group, $k\in H^4BG$ such that $\text{tr}(k)\in H^3_{S^1}(G/G)$ is a strongly topologically regular, positive twist, and let $X$ be a $G$-space. Then Grojnowski's sheaves $^k(\textit{Ell}_G^\text{Groj})^*X$ can be recovered from the presheaf of groups on $\mathcal{L}(X/G)$ defined by $U\mapsto\ ^{\text{tr}(k)}\vec{K}_{S^1}^*(U)\otimes\mathbb{C}$.
\end{lem}

\begin{proof} It suffices to show that the groups $H_{Z(a)}^*(X^{a};\mathbb{C})$ and the Looijenga line bundle $\mathcal{L}^k$ can be recovered from the $\vec{K}_{S^1}^*(-)\otimes\mathbb{C}$ presheaf. 
For $a\in\mathbb{H}\times\mathfrak{t}_\mathbb{C}$ write $\phi_1(a)$ and $\phi_2(a)$ for the two components of $\phi(a)$, $Z(a_1)$ and $Z(a_2)$ for their centralizers in $G$, and $X^{a_1}$ and $X^{a_2}$ for their fixed point sets in $X$. Regard ordinary cohomology as $\mathbb{Z}/2$-graded.
By the completion theorem in twisted $K$-theory ([FHT0] Theorem 3.9)
$$H_{Z(a)}^*(X^{a};\mathbb{C})\simeq K^*_{Z(a_1)}(X^{a_1})^{\widehat{}}_{a_2}.$$
Recall that $\mathcal{L}(X/G)$ is a full subgroupoid of ($X\times G)/G$ (where $G$ acts diagonally by the given action on $X$ and by conjugation on $G$) on those objects $(x,g)$ such that $gx=x$. Note that $X^{a_1}/Z(a_1)\simeq (X^{a_1}\times\{a_1\})/Z(a_1$ is a subgroupoid of $\mathcal{L}(X/G)$. Fractional splittings $\psi_{x,g}:G_{x,g}\leftarrow S^1_d$ can always be extended in the $X$-direction since the $B\mathbb{Z}$-action only depends on the second factor. Hence there is a fractional splitting $\psi_{a_1}$ that gives an isomorphism
$$^{\text{tr}(k)}\vec{K}_{S^1}^*(X^{a_1}/Z(a_1))\simeq\  ^{\text{tr}(k)}K^*(X^{a_1}/Z(a_1))((q)).$$
The twist $\text{tr}(k)$, which is pulled back from $\mathcal{L}(\text{pt}/G)$ defines a trivial $X^{a}$-parametrized family of $U(1)$-central extensions of $Z(a)$ (c.f. Lemma \ref{twistfamily}). Restricting to the fiber over $a_2\in Z(a)$ defines a trivial $U(1)$-principle bundle over $X^{a}$, whose associated trivial line bundle is denoted $\mathcal{L}^{\text{tr}(k)_{a}}$ and whose fiber will be denoted by $\mathbb{C}({a})$. Extract a single $q$-power from the right side of the previous display, complexify, and and apply the completion theorem in twisted $K$-theory ([FHT0] Theorem 3.9) to obtain
$$^{\text{tr}(k)}K^*(X^{a_1}/Z(a_1))\otimes\mathbb{C}^{\ \widehat{}}_{a_2} \simeq H_{Z(a)}^*(X^{a};\mathcal{L}^{\text{tr}(k)_{a}})\simeq H_{Z(a)}^*(X^{a};\mathbb{C}(a)).$$
To finish the proof it suffices to show that the vector spaces $\mathbb{C}(a)$ assemble into the Looijenga line bundle $\mathcal{L}^k$ over $T\times T$ as $a$ (or really $\phi(a)$) varies. This is clear since since $\mathbb{C}(a)$ is the fiber over the morphism $\phi(a)\in T\times T\subset(G/G/B\mathbb{Z})_1$ of the line bundle associated to the $U(1)$-extension defined by the twist $\tau$. 
\end{proof}

When $G$ is a torus or simple and simply-connected, Kitchloo [K] has constructed an equivariant elliptic cohomology theory over the complex numbers by defining certain $W\ltimes\pi_1T^2$-equivariant holomorphic sheaves $^k\mathcal{K}_{LG\rtimes S^1}^*LX$ on $\mathbb{H}\times\mathfrak{t}_\mathbb{C}$ as follows. Let $\tau \in H^4BG\simeq\mathbb{Z}$ be an element whose transgression in $H^3_{S^1}(G/G)$ defines a strongly topologically regular, positive twist. Let $\mathcal{F}_\tau^G$ be the space of Fredholm operators on a Hilbert space $\mathcal{H}_\tau$ which is the direct sum of countably many copies of each irreducible level $\tau$ positive energy representation of $LG$. Let $\mathbb{H}$ be the upper half plane. For a finite $G$-CW space $X$, define the sheaves $^\tau\mathcal{K}_{S^1\times T}^*LX$ on $\mathbb{H}\times\mathfrak{t}_\mathbb{C}$ by sheafifying the presheaves
$$U\mapsto \pi_{*\text{mod}2}\text{Maps}(U\times LX,\mathcal{F}_k^G)^T.$$
Then the $^\tau\mathcal{K}_{LG\rtimes S^1}^*LX$ are defined to be the sheaves of $\mathcal{O}(\mathbb{H}\times\mathfrak{t}_\mathbb{C})$-modules whose stalks at a point $(h,a)$ are the following inverse limits over finite $S^1\times T$-CW subspaces of $LX$:
$$\varprojlim_{Y \subset LX }(^\tau\mathcal{K}_{S^1\times T}^*Y)_{(h,a)}\otimes_{R(S^1\times T)}\mathcal{O}(\mathbb{H}\times\mathfrak{t}_\mathbb{C})_{(h,a)}.$$

\begin{lem} Let $G$ be a torus or simple and simply-connected Lie group, $0<\tau\in H^4BG$ and $X$ a $G$-space. Then Kitchloo's sheaves $^k\mathcal{K}_{LG\rtimes S^1}^*LX$ can be recovered from the the presheaf of groups on $\mathcal{L}(X/G)$ defined by $U\mapsto\ ^{\text{tr}(k)}\vec{K}_{S^1}^*(U)\otimes\mathbb{C}$
\end{lem}

\begin{proof} Similarly to the above, let $Y^a\subset Y$ denote the common fixed point set for $\phi_1(a)$, $\phi_2(a)$, and $S^1$ in $S^1\times T$. By Theorem 3.3 of [K], the stalks defined above are isomorphic to the groups
$$\varprojlim_{Y \subset LX }(^\tau K_{S^1\times T}^*Y^a)\otimes_{R(S^1\times T)}\mathcal{O}(\mathbb{H}\times\mathfrak{t}_\mathbb{C})_{(h,a)}.$$
The proof is now complete since the classical twisted $K$-theory groups $(^\tau K_{S^1\times T}^*Y^a)\simeq\ ^\tau K_T^*(Y^a)[q^{\pm}]$ can certainly be recovered from the indicated $\vec{K}_{S^1}^*$-theory presheaf as in the proof of Lemma \ref{grojcomp}.
\end{proof}

\section{Appendix}

\subsection{Example: $U(1)$} The groups $^{\tau}\vec{K}^{*}_{S^1}(U(1)/U(1))$ can be calculated directly. Note that the twist $\tau$ is an element of $H^3_{U(1)}U(1)\simeq\mathbb{Z}$. For twists transgressed from $H^4BU(1)$, i.e. classes in $2\mathbb{Z}$, the calculation is due to Constantin Teleman (unpublished notes). Here is an approach that covers all strongly topologically regular twists $\tau\neq0$.

The general idea us to apply Mayer-Vietoris to the usual decomposition of $U(1)$ into two overlapping arcs $U$ and $V$ which are equivariantly contractible and whose intersection has equivariantly contractible connected components. The automorphism group of each of these four components is a $U(1)$-central extension of a $U(1)$-extension of $S^1$ (c.f. Example \ref{torustwist}, Lemma \ref{twistfamily}), which is non-canonically isomorphic to $U(1)\times U(1)\times S^1$. Applying Lemma \ref{stalksofKhat} and Bott periodicity the Mayer-Vietoris sequence is a 6-term hexagon
$$\begin{tikzcd}
  &R_\star(U(1))((q))^{\oplus2}\arrow[r,"i^*-M^\tau"] &R_\star(U(1))((q))^{\oplus2} \arrow[dr]& \\
  ^\tau\vec{K}_{S^1}^0(U(1)/U(1))\arrow[ur]& & &\arrow[dl]^\tau\vec{K}_{S^1}^1(U(1)/U(1)) \\
  &\arrow[ul]0 &\arrow[l]0 &
\end{tikzcd}$$
exhibiting the desired groups as the kernel and cokernel of the top horizontal map. Write $R(U(1))\simeq\mathbb{Z}[t^\pm]$. From the explicit presentation of the twist $\mathcal{L}^\tau$ in Example \ref{rank1example} the identifications of the four $\vec{K}_{S^1}$-theory stalks with $R_\star(U(1))((q))$ can be chosen such that 3 of the restriction maps are the identity and the fourth sends an element $\sum\chi_k(t)q^k$ to $(qt)^\tau\sum\chi_k(qt)q^k$. Hence the map in question is
$$i^*-M^\tau=\begin{bmatrix}f(t,q)\\ g(t,q) \end{bmatrix}=\begin{bmatrix} f(t,q)-g(t,q)\\ g(t,q)-(tq)^{n}g(qt,q)\end{bmatrix}\in \text{End}\big(\mathbb{Z}[t^\pm]((q))\big)$$
When $\tau>0$ there is no cokernel, and the kernel has a basis
$$\sum_{k\in\mathbb{Z}}t^{j+nk}q^{n\frac{k(k-1)}{2}-kj}\ \ \ j=0,...,n-1$$ 
Some readers may notice that at $\tau=1$ this is the character of the basic level 1 positive energy representation of $LU(1)$ up to a factor of the (shifted) partition function: $\prod_{k>0}(1-q^{k})^{-1}$.

\section*{References}

\

\noindent
[A] M. Ando: Power Operations in Elliptic Cohomology \textit{Transactions of the AMS} \textbf{352} 2000

\

\noindent
[ABG] M. Ando, A. Blumberg, D. Gepner: Twists of K-Theory and TMF https://arxiv.org/pdf/1002.3004.pdf 

\

\noindent
[AS] M. F. Atiyah, I. M. Singer: Index theory for skew-adjoint Fredholm operators \textit{Inst. Hautes
Etudes Sci. Publ. Math.} \textbf{37} 1969

\

\noindent
[BK] M. Brion, S. Kumar: Frobenius Splitting Methods in Geometry and Representation Theory \textit{Birkhauser Progress in Mathematics} \textbf{231}

\

\noindent
[BET] D. Berwick-Evans, A. Tripathy: A de Rham model for complex analytic equivariant elliptic cohomology https://arxiv.org/pdf/1908.02868.pdf

\

\noindent
[Br] A. Broumas: Congruences Between the Coefficients of the Tate Curve via Formal Groups \textit{Pr. American Math. Soc.} \textbf{128} 1999

\

\noindent
[FHT0]  D. Freed, M. Hopkins, and C. Teleman: Twisted equivariant K-theory with complex coefficients \textit{Journal of Topology} \textbf{1}, 2007.

\

\noindent
[FHT1]  D. Freed, M. Hopkins, and C. Teleman: Loop Groups and Twisted K-Theory \textit{Journal of Topology} \textbf{4}, 2011

\

\noindent
[FHT2]  D. Freed, M. Hopkins, and C. Teleman: Loop Groups and Twisted K-Theory II \textit{Journal of Topology} \textbf{4}, 2011

\

\noindent
[FHT3]  D. Freed, M. Hopkins, and C. Teleman: Loop Groups and Twisted K-Theory \textit{Annals of Math} \textbf{174}, 2011.

\

\noindent
[FvdP] J. Fresnel, M. van der Put: Rigid Analytic Geometry and its Applications \textit{Birkhauser Progress in Mathematics} 2004

\

\noindent
[G] I. Grojnowski: Delocalised Equivariant Elliptic Cohomology in Elliptic Cohomology: Geometry, Applications, and Higher Chromatic Analogues \textit{Cambridge University Press}

\

\noindent
[GKV] V. Ginzburg, M. Kapranov, E. Vasserot: Elliptic Algebras and Equivariant Elliptic Cohomology I https://arxiv.org/pdf/q-alg/9505012.pdf

\

\noindent
[K] N. Kitchloo: Quantization of the modular functor and equivariant elliptic cohomology  https://arxiv.org/pdf/1407.6698.pdf

\

\noindent
[KM] N. Kitchloo, J. Morava: Thom Prospectra for Loopgroup Representations 	https://arxiv.org/pdf/ math/0404541.pdf

\

\noindent
[L] J. Lurie: A Survey of Elliptic Cohomology:   http://www.math.harvard.edu/~lurie/papers/survey.pdf

\

\noindent
[N] G. Nishida: The transfer homomorphism in equivariant generalized cohomology theories: \textit{J. Math. Kyoto Univ} \textbf{18-3} 1978

\

\noindent
[P] M. Papikan: Rigid analytic geometry and the uniformization of abelian varieties: http://www.math.psu.edu /papikian/Research/RAG.pdf

\

\noindent
[PS] A. Pressley, G. Segal: Loop Groups \textit{Oxford University Press} 1986

\

\noindent
[Pu] M. van der Put: Serre Duality for Rigid Analytic Spaces \textit{Indagationes Mathematicae} \textbf{3} 1992

\

\noindent
[Se] G. Segal:
The representation-ring of a compact Lie group
\textit{Publications mathématiques de l’I.H.É.S.} \textbf{34} (1968), p. 113-128

\

\noindent
[Se2] G. Segal, Elliptic Cohomology. S´eminaire Bourbaki 695 (1988) 187–201.

\

\noindent
[Se3] G. Segal:
What is an ellptic object?
\textit{Cambridge University Press}  (2010)

\

\noindent
[S] M. Spong: Comparison theorems for torus-equivariant elliptic cohomology theories \textit{Thesis, The University of Melbourne}

\

\noindent
[YT] Y. Tian: Introduction to Rigid Geometry https://www.math.uni-bonn.de/people/tian/Rigid.pdf

\end{document}